\documentclass[11pt]{article}
\usepackage[utf8]{inputenc}
\usepackage{blindtext}
\usepackage{soul}
\usepackage{latexsym} % special \LaTeX symbols
\usepackage{amsmath,amssymb,amsfonts,amsthm} % AMS math
\usepackage{graphics,psfrag} % graphics support 
\usepackage{fancyhdr,lastpage} % for making nice headers
\usepackage{color,verbatim} % text markup
\usepackage[hidelinks]{hyperref} % linking and navigation
\usepackage{hyperref} % linking and navigation
\usepackage[table]{xcolor} % make colored tables
\usepackage{metalogo}
\usepackage{etoolbox}
\usepackage{cite}
\usepackage{mathrsfs}
\usepackage[affil-it]{authblk}

\usepackage{xfrac}
\usepackage{bbm}
\patchcmd{\thebibliography}{\section*{\refname}}{}{}{}

\newtheorem{theorem}{Theorem}[section]
\newtheorem{corollary}{Corollary}[theorem]
\newtheorem{lemma}[theorem]{Lemma}

\newtheorem{definition}[theorem]{Definition}
\newtheorem{example}[theorem]{Example}
\newtheorem{remark}[theorem]{Remark}

\hypersetup{
    colorlinks=false,
    pdfborder={0 0 0},
}

\usepackage{calc}

\usepackage[a4paper,top=1in,bottom=0.8in,right=0.8in,left=0.8in]{geometry}
\let\oldbibliography\thebibliography
\renewcommand{\thebibliography}[1]{%
  \oldbibliography{#1}%
  \setlength{\itemsep}{-1.5mm}%
}

\def\R{\mathbb{R}}

\def\N{\mathbb{N}}
\def\P{\mathbb{P}}

\def\E{\mathbb{E}}

\newcommand{\be}{\begin{equation}}
\newcommand{\ee}{\end{equation}}
\newcommand{\bea}{\begin{eqnarray}}
\newcommand{\eea}{\end{eqnarray}}
\newcommand{\beann}{\begin{eqnarray*}}
\newcommand{\eeann}{\end{eqnarray*}}
\newcommand{\benn}{\begin{equation*}}
\newcommand{\eenn}{\end{equation*}}

\newcommand{\cB}{{\mathcal B}}  % calligraphic B
\newcommand{\cD}{{\mathcal D}}  % calligraphic D
\newcommand{\cL}{{\mathcal L}}  % calligraphic L
\newcommand{\cS}{{\mathcal S}}  % calligraphic S
\fancypagestyle{plain}{%
  \fancyhf{}%
  \fancyfoot[C]{\footnotesize Page \thepage\ of \pageref{LastPage}}%
}

\title{Action convergence of general hypergraphs and tensors}
\author{Giulio Zucal\thanks{giulio.zucal@mis.mpg.de}}
\affil[1]{Max Planck Institute for Mathematics in the Sciences, Leipzig, Germany}

\date{\today}

\begin{document}
\maketitle

\begin{abstract}

Action convergence provides a limit theory for linear bounded operators $A_n:L^{\infty}(\Omega_n)\longrightarrow L^1(\Omega_n)$ where $\Omega_n$ are potentially different probability spaces. This notion of convergence emerged in graph limits theory as it unifies and generalizes many notions of graph limits. We generalize the theory of action convergence to sequences of multi-linear bounded operators $A_n:L^{\infty}(\Omega_n)\times \ldots \times L^{\infty}(\Omega_n)\longrightarrow L^1(\Omega_n)$. Similarly to the linear case, we obtain that for a uniformly bounded (under an appropriate norm) sequence of multi-linear operators, there exists an action convergent subsequence. Additionally, we explain how to associate different types of multi-linear operators to a tensor and we study the different notions of convergence that we obtain for tensors and in particular for adjacency tensors of hypergraphs. We obtain several hypergraphs convergence notions and we link these with the hierarchy of notions of quasirandomness for hypergraph sequences. This convergence also covers sparse and inhomogeneous hypergraph sequences and it preserves many properties of adjacency tensors of hypergraphs. Moreover, we explain how to obtain a meaningful convergence for sequences of non-uniform hypergraphs and, therefore, also for simplicial complexes. Additionally, we highlight many connections with the theory of  dense uniform hypergraph limits (hypergraphons) and {\color{black}we conjecture the equivalence of this theory with a modification of multi-linear action convergence}.

    \vspace{0.2cm}
\noindent {\bf Keywords:}  Graph limits, Hypergraphs, Action convergence, Tensors, Higher-order interactions
    \vspace{0.2cm}
\noindent {\bf  Mathematics Subject Classification Number:}  05C65
\end{abstract}

\section{Introduction}

In the last 20 years, the study of complex networks has permeated many areas of social and natural sciences. Important examples are computer, telecommunication, biological, cognitive, semantic and social networks. In particular, in all of these areas, understanding large networks is a fundamental problem.

Network structures are usually modelled using graph theory to represent pairwise interactions between the elements of the network. However, for very large networks, such as the internet, the brain and social networks among others, exact information about the number of nodes and other specific features of the underlying graph is not available. For this reason, there is the need for a mathematical definition of synthetic structures containing only the relevant information for a very large graph. This is equivalent to assuming that the number of nodes is so big that the graph can be well approximated with a “graph-like” object with infinite number of nodes. This motivated the development of graph limits theory, the study of graph sequences, their convergence and their limit objects. In mathematical terms, one is interested in finding a metric on the space of graphs and a completion of this space with respect to this metric. This is a very active field of mathematics that connects graph theory with many other mathematical areas such as stochastic processes, ergodic theory, spectral theory and several branches of analysis and topology.  

From the rise of graph limits theory, two different cases have been mostly considered. The first case is the limits of dense graphs, i.e.\ when the number of edges of the graphs in the sequence is asymptotically proportional to the square of the number of vertices. This case, where the limit objects are called graphons (from graph functions), is now very well understood thanks to the contributions of L.\ Lovász, B.\ Szegedy, C.\ Borgs and J.\ Chayes among others  \cite{BORGS20081801,Lovsz2007SzemerdisLF, LOVASZ2006933}. The dense graph limit convergence is metrized by the cut-metric and is equivalent to the convergence of homomorphism densities. The completion of the set of all graphs in this metric is compact, i.e. every graph sequence has a convergent subsequence, which is a very useful property. A shortcoming of the dense graph limit theory is that it has not enough expressive power to study graph sequences in which the number of edges is sub-quadratic in the number of vertices. In fact, every sparse graph is considered to be similar to the empty graph in this metric. An important generalization of this theory are $L^p-$graphons \cite{LpGraphons1,Lpgraphon2}. The other case that has been well studied are graph sequences with uniformly bounded degree and the associated notion of convergence was introduced by I.\ Benjamini and O.\ Schramm \cite{BenjaminiLimit} and it has a stronger version called local-global convergence \cite{local-global1,Hatami2014LimitsOL}. The limits of such convergent sequences can be represented as objects called graphings. For a thorough treatment of these topics see the monograph by L.\ Lovász \cite{LovaszGraphLimits}. 

Unfortunately, for most applications, the really interesting case is the intermediate degree case, not covered by the previously presented theories. Real networks are usually sparse but not very sparse and heterogeneous. For this reason, the intermediate case attracted a lot of attention recently, see for example \cite{MarkovSpaces,KUNSZENTIKOVACS20191,KUNSZENTIKOVACS2022109284}.  In particular, in a recent work Á.\ Backhausz and B.\ Szegedy introduced a new functional analytic/measure theoretic notion of convergence \cite{backhausz2018action}, that not only covers the intermediate degree case but also unifies the graph limit theories previously presented. This notion of convergence is called action convergence and the limit objects for graph sequences are called graphops (from graph operators). More generally this is a notion of convergence for $P-$operators, i.e.\ linear bounded operators
$$
A:L^{\infty}(\Omega)\longrightarrow L^1(\Omega)
$$
where $\Omega$ is a probability space. As a matrix can be naturally interpreted as a $P-$operator, we obtain as a special case a notion of convergence for matrices. The notions of convergence for graphs are derived by associating to graphs (properly normalized) matrices, for example, adjacency matrices or Laplacian matrices. 
In this work we extend the notion of action convergence to multi-linear operators. More specifically, we consider multi-linear operators of the form 
$$
A:L^{\infty}(\Omega)^r\longrightarrow L^1(\Omega)
$$
where $\Omega$ is a probability space and $L^{\infty}(\Omega)^r=L^{\infty}(\Omega)\times\ldots\times L^{\infty}(\Omega)$ is the cartesian product of $L^{\infty}(\Omega)$ with itself $r$ times. We name such operators multi-$P-$operators. 

This convergence notion comes with an associated pseudo-metric $d_M$. We therefore say that two multi-$P-$operators $A$ and $B$ are isomorphic if $d_M(A,B)=0$. The space of classes of isomorphism of multi-$P-$operators equipped with $d_M$ is a metric space.

We obtain a compactness result for multi-$P-$operators analogous to the compactness result for the case of $P-$operators: Sequences of multi-$P-$operators $(A_n)_{n}$ that have a uniform bound $C>0$ on the quantity
$$
\|A_n\|_{p_1,\ldots,p_{r}\rightarrow q}=\sup_{0\neq f_1,\ldots, f_r\in L^\infty(\Omega_n)}\frac{\|A_n[f_1,\ldots,f_r]\|_q}{\|f_1\|_{p_1}\ldots \|f_r\|_{p_r}}\leq C
$$
for all $n\in \mathbb{N}$ have a convergent subsequence in the space of isomorphism classes of multi-$P-$operators equipped with the metric $d_M$.  Moreover,
$$
\|A\|_{p_1,\ldots,p_{r}\rightarrow q}\leq\lim_{n\rightarrow \infty}\|A_n\|_{p_1,\ldots,p_{r}\rightarrow q}\leq C.
$$
if the sequence is convergent with limit multi-$P-$operator $A$. 

We focus on using multi-linear action convergence to define meaningful convergence notions for tensors and hypergraphs.

\begin{definition}
Let $r,n\geq 2$. An $r$-th order $n$-dimensional \emph{tensor} $T$ consists of $n^r$ entries
\begin{equation*}
    T_{i_1,\ldots,i_r}\in \mathbb{R},
\end{equation*}
where $i_1,\dots,i_r\in[n]$.\newline 
The tensor $T$ is \emph{symmetric} if its entries are invariant under any permutation of their indices.\newline
\end{definition}
First of all, we explain how symmetric tensors can be associated with multi-$P-$operators in multiple ways. For example, for a $3-$rd order symmetric tensor 
$$
T_{i_1,i_2, i_3}
$$
we can consider the operator
$$
T_1[\mathrm{v},\mathrm{w}]=\sum^n_{i_1,i_2=1}T_{i_1,i_2, i_3}v_{i_1}w_{i_2}
$$
where $\mathrm{v}=(v_i)_i,\mathrm{w}=(w_i)_i\in \R^n$ are vectors or alternatively

$$
T_2[f,g]=\frac{1}{2}(\sum^n_{i_2=1}T_{i_1,i_2,i_3}f_{i_1,i_2}g_{i_2,i_3} +\sum^n_{i_2=1}T_{i_1,i_2,i_3}f_{i_3,i_2}g_{i_2,i_1})
$$
where $f=(f_{ij})_{ij},g=(g_{ij})_{ij}$ are symmetric matrices. These different choices of associating a multi-$P-$operator to a tensor give rise in general to different convergence notions for tensors. In the case of $3-$rd order symmetric tensors the second choice presented seems to be in many cases more appropriate. However, one can require action convergence of both multi-$P-$operators $(T)_1$ and $(T)_2$ associated to the tensor $T$ at the same time.

Recently, in network sciences, a lot of interest has been generated by higher-order interactions (interactions that are beyond pairwise) and the phenomena generated by them \cite{HypergraphsNetworks,HigerOrdIntBook,carletti2020dynamical,majhi2022dynamics,MHJ,HypergraphsDynamics,Bohle_2021,JOST2019870,JostMulasBook}. Hypergraphs are the natural mathematical/combinatorial structure to represent higher-order interactions. 

\begin{definition}
 An \emph{hypergraph} is a pair $H=(V,E)$ where $V=\{v_1,\ldots,v_n\}$ is the set of \emph{vertices}, $E=\{e_1,\ldots,e_m\}$ is the set of \emph{edges} and $\emptyset \neq e \subseteq V$ for each $e\in E$.\newline
 A hypergraph $H$ is $k$\emph{-uniform} if $|e|=k$ for every $e\in E$.
 \end{definition}

Limit theories for hypergraphs are much less developed than the ones for graphs due to the bigger combinatorial complexity and for this reason very limited to the uniform and dense hypergraph sequences case. The first contributions on hypergraph limits,\cite{hypergrELEK20121731,HypergraphsSzegedy2} by Elek and Szegedy used techniques from nonstandard analysis/model theory to define uniform and dense hypergraph limit objects. This approach using “ultralimits” is well explained in the recent book \cite{RandomneLimitTow} by Towsner. A more classical approach using “quotients” and regularity partitions obtaining the same type of limits has been developed by Zhao in  \cite{HypergraphonsZhao}. The limit objects of this convergence notions appeared earlier in the context of exchangeable arrays of random variables\cite{kallenberg1992symmetries, hoover1979relations, aldous1981representations, aldous2010exchangeability,austin2008exchangeable,diaconis2007graph}. For sparse uniform hypergraph sequences, a removal lemma is obtained in \cite{SparseHypLogicLimit} using again techniques from logic but no limit theory/convergence for sparse hypergraphs is developed to the best of our knowledge. Our hypergraphs convergence based on multi-linear action convergence instead is based on functional analytic and measure-theoretic techniques and can be applied to any hypergraph sequence, also for non-uniform, very sparse and heterogeneous hypergraphs sequences. 

To apply action convergence we associate to hypergraphs their adjacency tensor

\begin{definition}
Let $H=(V,E)$ be a hypergraph on $n$ nodes with largest edge cardinality $r$. The \emph{adjacency tensor} of $H$ is the $r$-th order $n$-dimensional tensor $A=A(H)$ with entries
\begin{equation*}
A_{i_1,\ldots,i_r}:=\begin{cases}
0 & \text{ if }\{v_{i_1},\ldots,v_{i_r}\}\notin E\\
1 &  \text{ if }\{v_{i_1},\ldots,v_{i_r}\}\in E.\\
\end{cases}
\end{equation*}
\end{definition}

(possibly multiplied by some normalizing constant) and as already explained we can associate a tensor with different multi-$P-$operators and therefore different convergence notions with a relationship between them. These different types of convergence are related to the different types of quasi-randomness for sequences of hypergraphs. In particular, we focus our attention on one notion of convergence obtained in such a way, which we consider being in many cases the most appropriate, and we compare it with the existing notion of convergence for dense hypergraphs (hypergraphon convergence). We underline many similarities in the two theories and we look at some motivating examples that bring us {\color{black} to conjecture that a modification of action convergence of the normalized adjacency tensor and hypergraphon convergence are equivalent. }

The generalization of action convergence to multi-linear operators allows us to study hypergraph limits and therefore to represent conveniently large hypernetworks with objects that we will call hypergraphops. Hypergraphops are symmetric and positivity-preserving multi-$P-$operators and hypergraphs are obviously special cases of hypergraphops. We show that the space of hypergraphops (with a uniform bound on some operator norm) is closed. In fact, symmetry and positivity of multi-$P-$operators are preserved under action convergence, i.e.\ the limit of an action convergent sequence of symmetric, respectively positivity-preserving, multi-$P-$operators is again symmetric, respectively positivity-preserving.\newline

{\color{black}We compare the action convergence metric with other norms and metrics in order to better understand this convergence. These comparisons allow us to give several examples of action convergent sequences of hypergraphs and their limits.  }\newline

We also study other possible tensors associated with hypergraphs and their associated action convergence. In particular, we present possible choices to obtain meaningful limit objects for inhomogeneous and non-uniform hypergraph sequences. In particular, to the best of our knowledge, we are the first to introduce a meaningful convergence for non-uniform hypergraphs. Covering the case of non-uniform hypergraphs, our limit theory gives us a convergence for simplicial complexes as a special case answering a question in \cite{Bobrowski2022}.

Generalising the results of \cite{backhausz2018action} is technically challenging as it requires us to use multi-linear operators and tensors instead of linear operators and matrices, for which there are fewer results. Furthermore, this generalisation also significantly complicates the associated notation. However, on a more conceptual level, the main challenge of understanding the limit objects of hypergraphs requires a deeper understanding of action convergence to which we contribute here.

\subsection*{Structure of the paper}

In Section \ref{SecNotation} we introduce the notation and basic definitions from functional analysis and probability theory. In Section \ref{SecActionConve} we briefly recall the theory of action convergence and in Section \ref{SecTensHyp} we introduce relevant notions for hypergraphs and tensors. In Section \ref{SecMultiActConv} we finally introduce the generalization of  action convergence to multi-linear operators and in Section \ref{SecConstrLimObj} we prove the main compactness result. Moreover, in Section \ref{SecPropLimObje} we study several properties of multi-linear action convergence and of the related limit objects. {\color{black} In Section \ref{SecNormsMulti} we compare the action convergence distance with other norms and distances for multi-linear operators.} In Section \ref{SecMultActTensHyp} and \ref{NonUnifHypSect} we investigate how action convergence for multi-linear operators can be specialized to tensors and hypergraphs in different ways and we study the different convergence notions obtained. In Section \ref{SectHypergraphons} we point out many relationships between  {\color{black}hypergraph convergence notions obtained from action convergence and hypergraphon convergence in the context of dense hypergraph sequences and we conjecture the equivalence of a modification of action convergence and hypergraphon convergence.}  

\section{Notation}\label{SecNotation}

In the following, we will denote with $(\Omega,\mathcal{F},\P)$ a standard probability space where $\mathcal{F}$ is a  $\sigma-$algebra and $\P$ is a probability measure on $(\Omega,\mathcal{F})$. We will denote with $\mathcal{P}(\Omega,\mathcal{F})$ or shortened $\mathcal{P}(\Omega)$ the set of probability measures on $(\Omega,\mathcal{F})$. Moreover, we will indicate the expectation of a real-valued measurable function (or in probabilistic language a random variable) $f$ on $(\Omega,\mathcal{F},\P)$ with $\E[f]$. We indicate the (possibly infinite) $L^p-$norm of a real-valued measurable function $f$ with  
$$
\|f\|_p=\left(\int_{\Omega}|f(\omega)|^pd\P(\omega)\right)^{1/p}=\left(\E[|f|^p]\right)^{1/p}.
$$ If a measurable function $f$ has finite $L^p-$norm we say that $f$ is $p-$integrable (or has finite $p-$moment). We denote with $L^p(\Omega,\mathcal{F},\P,)$ the usual Banach space of the real-valued measurable $p-$integrable functions (identified if they are equal almost everywhere) on $(\Omega,\mathcal{F},\P)$ equipped with the usual $L^p-$norm or equivalently, in probabilistic language, the space of random variables with finite $p-$moment. We will use a lot of times the shortened notations $L^p(\Omega)$ or $L^p$ when there is no risk of confusion. For a set $S\subset \R$ we will also denote with $L_S^p(\Omega)$ the space of the $p-$integrable random variables taking values in $S$.\newline 

For a linear operator 
$$\begin{aligned}
    A:L^p(\Omega, \mathcal{F},\P)&\longrightarrow L^q(\Omega,\mathcal{F},\P)\\
& f\mapsto Af
\end{aligned}$$ we define the $(p,q)-$operator norm $$
\|A\|_{p\rightarrow q}=\sup_{f\in L^p,\ f\neq 0} \frac{\|Af\|_q}{\|f\|_p}.
$$
The linear operator $A$ is said to be bounded (or equivalently continuous) if the operator norm is finite. We denote with $\cB_{p,q}$ the Banach space of linear bounded operators from $L^p(\Omega)$ to $L^q(\Omega)$ equipped with the $(p,q)-$operator norm.\newline
A $k-$dimensional random vector is a measurable function $\mathbf{f}$ from a probability space $(\Omega,\mathcal{F},\P)$ to $\R^k$ and we can naturally represent it as $$
\mathbf{f}=(f_1,\ldots,f_k)
,$$ where $f_1,\ldots,f_k$ are real-valued random variables on $(\Omega,\mathcal{F},\P,)$. Therefore, a real-valued random variable is a $1-$dimensional random vector. For a $k-$dimensional random vector $\mathbf{f}$, we denote with $\mathcal{L}(\mathbf{f})=\mathcal{L}(f_1,\ldots,f_k)$ its distribution (or law), that is the measure on $\R^k$ defined as  
$$
\mathcal{L}(\mathbf{f})(A)=\P(\mathbf{f}^{-1}(A))
$$
where $A$ is a set in the Borel $\sigma$-algebra of $\R^k$. \newline
Given $n\in\N$,
we denote by $[n]$ the set $\{1,\ldots,n\}$. In the case of a finite probability space, the law of a random vector has a particularly easy representation. We show this with the following example that will be important in the next sections. 
\begin{example}
Let's consider the probability space $([n],\mathcal{D},\mathcal{U})$ where $\mathcal{U}$ is the uniform probability measure on $[n]$ and with $\mathcal{D}$ the discrete $\sigma-$algebra on $[n]$. Then for any $k-$dimensional random vector $$
\mathbf{f}=(f_1,\ldots,f_k)$$
the law $\mathcal{L}(\mathbf{f})$ is 
$$
\mathcal{L}(\mathbf{f})=\frac{1}{n}\sum^n_{i=1}\delta_{(f_1(i),\ldots,f_k(i))}
$$
where $\delta_{(x_1,\ldots,x_k)}$ is the Dirac measure centered in $(x_1,\ldots,x_k)\in \R^k$.

\end{example}

\section{Action convergence}\label{SecActionConve}
We briefly recall here, following \cite{backhausz2018action}, the notion of action convergence of operators, a very general notion of convergence for operators acting on $L^p$ spaces defined on different probability spaces, introduced in the context of graph limit theory. Other related works to this limit notion are \cite{ArankaAction2022,MeasTheorActionZucal}.\newline
We start giving the following definition.
\begin{definition}
A $P-$\emph{operator} is a linear bounded operator $$\begin{aligned}
     A:L^\infty(\Omega, \mathcal{F},\P)&\longrightarrow L^1(\Omega,\mathcal{F},\P)
\end{aligned}
$$
for any probability space $(\Omega, \mathcal{F},\P)$.\newline
A $P-$operator $A$ is \emph{acting} on the probability space $(\Omega,\mathcal{F},\P)$ if the $L^1$ and $L^\infty$ spaces are defined on $(\Omega, \mathcal{F},\P)$. We denote the set of all $P-$operators with $\mathcal{B}$ and the set of all $P-$operators acting on $(\Omega,\mathcal{F},\P)$ with $\mathcal{B}(\Omega, \mathcal{F},\P)$.
\end{definition}

We give here an example that will be important in the following.
\begin{example}
A matrix $A=(A_{i,j})_{i,j\in [n]}$ can be interpreted as a $P-$operator acting on the probability space $\Omega=([n],\mathcal{D},\mathcal{U})$ where we denoted with $\mathcal{U}$ the uniform probability measure on $[n]$ and with $\mathcal{D}$ the discrete $\sigma$-algebra on $[n]$. In fact, for $\mathrm{v}=(v_i)_{i\in [n]}\in \R^n\cong L^{\infty}(\Omega)\cong L^1(\Omega)$
$$
A:L^{\infty}(\Omega)\longrightarrow L^1(\Omega)
$$

$$
(A\mathrm{v})_i=\sum^n_{j=1}A_{ij}v_j.
$$ In particular, a graph can be associated to its adjacency matrix (or its Laplacian matrix) and therefore it can be interpreted as a $P-$operator.
\end{example}

We would now like to introduce a metric on $P-$operators possibly acting on different probability spaces. This means that we would like to equip $\mathcal{B}$ with a metric and, therefore, with a natural notion of convergence. In reality, we will equip $\mathcal{B}$ with a pseudo-metric and then quotient over equivalent classes (elements at distance $0$) of $\mathcal{B}$ to obtain a proper metric space.  We will see that for graphs (adjacency matrices of graphs) this identification of elements is exactly what we want as it identifies isomorphic graphs.\newline
By definition, an element $f$ of $L^\infty(\Omega,\mathcal{F},\P)$ is a real-valued bounded random variable on $(\Omega,\mathcal{F},\P)$. Therefore, for a $P-$operator $A$ acting on $(\Omega,\mathcal{F},\P)$, 
$$
Af\in L^1(\Omega,\mathcal{F},\P) $$is, by definition, a real-valued random variable with finite expectation. Therefore, for functions $f_1,\ldots,f_k\in L^{\infty}(\Omega)$ we can consider the $2k-$dimensional random vector
$$
(f_1,Af_1,\ldots , f_k,Af_k)
$$and in particular its distribution $\mathcal{L}(f_1,Af_1,\ldots , f_k,Af_k)\in \mathcal{P}(\R^{2k})$. For a $P-$operator $A$, if a measure $\mu\in\mathcal{P}(\R^{2k})$ is such that $$
\mu= \mathcal{L}(f_1,Af_1,\ldots , f_k,Af_k)
$$
for some functions $f_1,\ldots, f_k\in L^{\infty}(\Omega)$ we say that $\mu$ is a measure generated by $A$ through $f_1,\ldots, f_k$.
We now define the set of measures generated by $A$. For reasons that will be clear in the following, it will be convenient to allow in our sets only measures generated by functions in $L_{[-1,1]}^{\infty}(\Omega)$, i.e. functions taking values between $+1$ and $-1$ almost everywhere. Therefore, we define the $k-$\emph{profile} of $A$,  $\cS_k(A)$, as the set of measures generated by $A$ through functions in  $L_{[-1,1]}^{\infty}(\Omega)$, i.e.\
\begin{equation}\label{EqDefKprofPoper}
\cS_k(A)=\bigcup_{f_1,\ldots, f_k\in L_{[-1,1]}^{\infty}(\Omega)}\{\mathcal{L}(f_1,Af_1,\ldots , f_k,Af_k)\}.
\end{equation}

This is a set of measures. To compare sets of measures, first of all, we will need  a metric on the space of measures.  For this reason, we recall the following well-known metric:

\begin{definition}[Lévy-Prokhorov metric]\label{LevyProk}
 The \emph{Lévy-Prokhorov Metric} $d_{\mathcal{LP}}$ on the space of probability measures $\mathcal{P}\left(\mathbb{R}^{k}\right)$ is for $\eta_1,\eta_2\in \mathcal{P}\left(\mathbb{R}^{k}\right)$
$$\begin{aligned}
d_{\mathcal{LP}}\left(\eta_{1}, \eta_{2}\right)=&\inf \left\{\varepsilon>0: \eta_{1}(U) \leq \eta_{2}\left(U^{\varepsilon}\right)+\varepsilon \text{ and } \right.\\
&\left.\eta_{2}(U) \leq \eta_{1}\left(U^{\varepsilon}\right)+\varepsilon  \text{ for all } U \in \mathcal{U}_{k}\right\},
\end{aligned}$$

where $\mathcal{U}_{k}$ is the Borel $\sigma$-algebra on $\mathbb{R}^{k}$, $U^{\varepsilon}$ is the set of points that have Euclidean distance smaller than $\varepsilon$ from $U$.
\end{definition}

The above metric metrizes the weak/narrow convergence for measures. \newline
We want to be able to compare sets of measures. We, therefore, introduce the following (pseudo-)metric on the sets of measures.

\begin{definition}[Hausdorff metric]\label{DefHausdorffDist}
 Given $X, Y\subset \mathcal{P}\left(\mathbb{R}^{k}\right)$, their \emph{Hausdorff distance} 
$$
d_{H}(X, Y):=\max \left\{\sup _{x \in X} \inf _{y \in Y} d_{\mathcal{LP}}(x, y), \sup _{y \in Y} \inf _{x \in X} d_{\mathcal{LP}}(x, y)\right\}.
$$

\end{definition}
Note that $d_{H}(X, Y)=0$ if and only if $\operatorname{cl}(X)=\operatorname{cl}(Y)$, where $\operatorname{cl}$ is the closure in $d_{\mathcal{LP}} .$ It follows that $d_{H}$ is a pseudometric for all subsets in $\mathcal{P}\left(\mathbb{R}^{k}\right)$, and it is a metric for closed sets.

Moreover, observe that by definition, the Lévy-Prokhorov distance between probability measures is upper-bounded by $1$ and, therefore, the Hausdorff metric for sets of measures is upper-bounded by  $1$.\newline

We are now ready to define the pseudo-metric we are interested in.
Consider two $P-$operators
$$
A: L^{\infty}(\Omega_1,\mathcal{F}_1,\P_1)\rightarrow L^1(\Omega_1,\mathcal{F}_1,\P_1)
$$and
$$
B: L^{\infty}(\Omega_2,\mathcal{F}_2,\P_2)\rightarrow L^1(\Omega_2,\mathcal{F}_2,\P_2).
$$
\begin{definition}[Metrization of action convergence]\label{DefActMetric} For the two $P$-operators $A, B$ the \emph{action convergence metric} is
$$
d_{M}(A, B):=\sum_{k=1}^{\infty} 2^{-k} d_{H}\left(\cS_{k}(A), \cS_{k}(B)\right).
$$
\end{definition}

Moreover, we will say that a sequence of $P$-operators $\left\{A_i \in \mathcal{B}\left(\Omega_i\right)\right\}_{i=1}^{\infty}$ is a Cauchy sequence if the sequence is Cauchy in $d_M$. \newline

The metric $d_M$ has some nice compactness properties. In particular, the following theorem gives us that sets of $P-$operators with uniformly bounded $(p,q)-$norm with $p\neq \infty$ are pre-compact in the action convergence metric.

\begin{theorem}[Theorem 2.14 in \cite{backhausz2018action}]\label{THMCompBack} Let $p\in [1,\infty)$ and $q\in [1,\infty]$. Let $\{A_i\}_{i=1}^\infty$ be a Cauchy sequence of $P$-operators with uniformly bounded $\|\cdot\|_{p\to q}$ norms. Then there is a $P$-operator $A$ such that $\lim_{i\to\infty} d_M(A_i,A)=0$, and $\|A\|_{p\to q}\leq\sup_{i\in\mathbb{N}}\|A_i\|_{p\to q}$. Therefore, the sequence $\{A_i\}_{i=1}^\infty$ is action convergent. 
\end{theorem}

Importantly, we can relate action convergence to the notions of convergence arising in the cases of dense graph sequences convergence (cut-metric, graphons) and uniformly bounded degree graph sequences convergence (local-global convergence). 

In particular, consider the sequence of adjacency matrices $A_n$ of graphs $G_n$, and let $v_n$ be the number of vertices of $G_n$. Then,  \begin{itemize}
    \item The action convergence of the sequence $$
   \frac{A_n}{v_n} 
    $$
    coincides with graphon convergence \cite[Theorem 8.2 and Lemma 8.3]{backhausz2018action}
    \item The action convergence of the sequence $$
    A_n
    $$
    coincides with local-global convergence \cite[Theorem 9.2]{backhausz2018action}.\end{itemize}
    We refer to \cite{backhausz2018action} for more details.

A lot of properties of matrices and graphs can be directly translated into the language of $P-$operators: self-adjointness, positivity, positivity-preservation and regularity, see Definition 3.1 in \cite{backhausz2018action}.

All these properties carry over in the limit if we assume a uniform bound on the $(p,q)-$norm with $p,q\notin\{1,\infty\}$. For the rigorous results, see Lemma 3.2 in \cite{backhausz2018action}, Proposition 3.4 in\cite{backhausz2018action}, Corollary 2.2 in \cite{ArankaAction2022}, where counterexamples are provided for when the assumptions of the Corollary 2.2 are not satisfied, and Proposition 3.4 in \cite{backhausz2018action}. 

In particular, the following special class of $P-$operators is important in graph limits theory. 
\begin{definition}
A positivity-preserving and self-adjoint $P-$operator
is called a \emph{graphop}.
\end{definition}

A graphop is the natural translation in the language of $P-$operators of the adjacency matrix of a graph when we consider the uniform probability measure on the nodes.

\section{Tensors and hypergraphs}\label{SecTensHyp}

We start by giving some preliminary definitions and notations on tensors and hypergraphs.\newline 

We indicate a vector in $\R^n$ by $\textbf{x}=(x_1,\dots,x_n)$. For a set $E$ we denote with $|E|$ the cardinality of $E$ and with $2^E$ the powerset of $E$.\newline
\begin{definition}
The \emph{symmetrization} of an  $r-$th order tensor $T$ is the $r$-th order tensor $Sym(T)$ where
$$
(Sym(T))_{i_1\ldots,i_r}=\frac{1}{r!}\sum_{\sigma \in \Sigma}T_{i_{\sigma(1)},\ldots,i_{\sigma(r)}},
$$
where $\Sigma$ is the set of all permutations of $[r]$.
\end{definition}

It will be very convenient in the following to consider a tensor as many different possible operators. 
\begin{definition}\label{DefAction}
For an $r-$th order $n-$dimensional symmetric tensor $T$ and for $s\in [r-1]$, the \emph{$s-$action} of $T$ on the $s-$th order $n-$dimensional symmetric tensors $f^{(1)},\ldots, f^{(r-1)}$ is the operation 
\begin{equation*}
\begin{aligned}
&(T[f^{(1)},\ldots,f^{(r-1)}])_{i_1,\ldots,i_s}=\\
Sym(\sum^n_{j_1,\ldots,j_{r-s}=1}T_{j_1,\ldots,j_{r-s},i_1,\ldots,i_{s}}&f^{(1)}_{i_2,\ldots,i_{s},j_1}f^{(2)}_{i_3,\ldots,i_{s},j_1,j_2}\ldots f^{(r-2s+1)}_{j_{r-2s+2},\ldots,j_{r-s},i_1}\ldots f^{(r-1)}_{j_{r-s},i_1,\ldots,i_{s-1}}).
\end{aligned}
\end{equation*}

\end{definition}
The $s-$action of $T$ is an operator that sends $r-1$ $s-$th order $n-$dimensional  symmetric tensors in an $s-$th order $n-$dimensional symmetric tensor. Therefore, the $s-$action is an operator acting on real-valued functions with domain the set of subsets of cardinality $s$ of $[n].$ 

To make the definition more clear, we give here some examples of $s-$action of a tensor that we will also use in the following.

\begin{example}
For an $r-$th order $n-$dimensional symmetric tensor $T$ the $1-$action of $T$ on the $n-$dimensional vectors  (first-order tensors) $f^1,\ldots, f^{r-1}\in \R^n$  is the operation
\begin{equation*}
(T[f^{(1)},\ldots,f^{(r-1)}])_{i}=\sum^n_{j_2,\ldots j_{r-1}=1}T_{i,j_2,\ldots,j_{r-1}}f^{(1)}_{j_2}\ldots f^{(r-1)}_{j_{r-1}}.
\end{equation*}

In the case of $r=2$, a second-order tensor $T$ is a matrix and the $1-$action on a vector $f$ is just  the classical matrix multiplication with the vector $f\in \R^n$, i.e.
\begin{equation*}
(Tf)_{i}=\sum^n_{j=1}T_{ij}f_{j}.
\end{equation*}

\end{example}

\begin{example}
For an $r-$th order $n-$dimensional symmetric tensor $T$ and the $(r-1)-$action of $T$ on the $s-$th order $n-$dimensional symmetric tensors $f^{(1)},\ldots, f^{(r-1)}$ is the operation 
\begin{equation*}
\begin{aligned}
(T[f^{(1)},\ldots &,f^{(r-1)}])_{i_1,\ldots,i_{r-1}}\\
&=Sym(\sum^n_{j=1}T_{j,i_2,\ldots,i_{r}}f^{(1)}_{j,i_1,i_2,\ldots,i_{s-1},\hat{i}_s}\ldots f^{(p)}_{j,i_1,i_2,\ldots,  \hat{i}_p\ldots i_{s}} \ldots f^{(r-1)}_{j,\hat{i}_1,i_2,i_3\ldots,i_{s}}).
\end{aligned}
\end{equation*}

    In particular, for a third-order $n-$dimensional symmetric tensor $T$ the $2-$action of $T$ on the $n\times n$ symmetric matrices $f=(f_{i,j})_{i,j\in [n]}$ and $g=(g_{i,j})_{i,j\in [n]}$ is the operation
\begin{equation*}
(T[f,g])_{i,k}=\frac{1}{2}(\sum^n_{j=1}T_{j,i,k}f_{j,i}g_{j,k}+\sum^n_{j=1}T_{j,k,i}f_{j,k}g_{j,i} ).
\end{equation*}
\end{example}

{\color{black}
\begin{remark}
For an $r-$th order $n-$dimensional (not necessarily symmetric) tensor $T$ and for $s\in [r-1]$, we can also consider the \emph{non-symmetrized $s-$action} of $T$ on the $s-$th order $n-$dimensional (not necessarily symmetric) tensors $f^{(1)},\ldots, f^{(r-1)}$ is the operation 
\begin{equation*}
\begin{aligned}
&(T[f^{(1)},\ldots,f^{(r-1)}])_{i_1,\ldots,i_s}=\\
\sum^n_{j_1,\ldots,j_{r-s}=1}T_{j_1,\ldots,j_{r-s},i_1,\ldots,i_{s}}&f^{(1)}_{i_2,\ldots,i_{s},j_1}f^{(2)}_{i_3,\ldots,i_{s},j_1,j_2}\ldots f^{(r-2s+1)}_{j_{r-2s+2},\ldots,j_{r-s},i_1}\ldots f^{(r-1)}_{j_{r-s},i_1,\ldots,i_{s-1}}.
\end{aligned}
\end{equation*}

\end{remark}}

We now introduce some notation and notions for hypergraphs.

Given an edge $e\in E$, we recall that we denote its cardinality by $|e|$, and in the following we will denote with $r$ the maximal edge cardinality, i.e.\
\begin{equation*}
  r:=\max_{e\in E}|e|.
\end{equation*}
Moreover, we observe that for the adjacency tensor of a hypergraph $H=(V,E)$ on $n$ vertices with largest edge cardinality $r$ its adjacency tensor
\begin{equation*}
A_{i_1,\ldots,i_r}:=\begin{cases}
0 & \text{ if }\{v_{i_1},\ldots,v_{i_r}\}\notin E\\
1 &  \text{ if }\{v_{i_1},\ldots,v_{i_r}\}\in E.\\
\end{cases}
\end{equation*}
is a standard notion for $r-$uniform hypergraphs. However, also edges with non-maximal cardinality are incorporated as repeated indices correspond to sets of lower cardinality. 

We give here some examples of deterministic and random hypergraphs. We will use these examples in the following.

{\color{black}

\begin{example}
 The \emph{complete} $r-$\emph{uniform} hypergraph on $n$ vertices is the hypergraph with $[n]$ and such that $E$ is the set of all ${n \choose r}$  subsets of $V$ with cardinality $r$. 
\end{example}

Recall that graphs are the $2-$uniform hypergraphs. Therefore, a random graph is a $2-$uniform random hypergraph. {\color{black} We recall here the Erd\"os-Renyi random graph model.

\begin{example}[Erd\"os-Renyi graph]
Consider the vertex set $V=[n]$ and we connect each of the possible ${n \choose 2}$ pairs independently with probability $p$, i.e.\ following the law of independent Bernoulli random variables. This is the Erd\"os-Renyi random graph and we will denote it with $G(n,p)$.

\end{example}}
}

A very common random uniform hypergraph that we will consider is the following.

\begin{example}[$r-$uniform Erdős–Rényi random hypergraph] \label{ERRandomHypergraph}We denote with $G(n,p,r)$ the $r-$uniform random hypergraph with vertex set $V=[n]$ and with edge set $E$ defined as follows: For every $e$, set of vertices of cardinality $r$,  $e$ is in $E$ with probability $p$, i.e.\ every edge of cardinality $r$ is in $E$ following independent Bernoulli random variables with parameter p. In the case r=2, $G(n,p,2)$ corresponds with the Erdős–Rényi random graph $G(n,p)$.
\end{example}

In the case $p=1$ the $r-$uniform Erdős–Rényi random hypergraph $G(n,1,r)$ corresponds with the complete $r-$uniform hypergraph.

We give also another example of uniform random hypergraph:
\begin{example}\label{TriangRandHyp}
We denote with $T(n,p)$ the random $3-$uniform hypergraph constructed taking the vertex set $V=[n]$  and as edges the triangles of the  Erdős–Rényi random graph $G(n,p)$ on the same vertex set $V=[n]$.  
\end{example}

We can generalize naturally this random hypergraph model
\begin{example}\label{RandHypergGener}
We denote with $R(n,p_1,\ldots p_{r-1},r)$ the $r-$uniform random hypergraph on the vertex set $V=[n]$ constructed inductively on $r$ as follows:
\begin{itemize}
\item for $r=2$ we define $R(n,p,2)=G(n,p,2)$.  
\item for $r>2$ we define $R(n,p_1,\ldots,p_{r-1},r) $  as the $r-$uniform hypergraph constructed selecting as edges independently with probability $p_{r-1}$ the sets of  $r$ vertices such that $R(n,p_1,\ldots,p_{r-2},r-1) $ restricted to these $r$ vertices is the $(r-1)-$uniform complete hypergraph on $r$ vertices. 
 \end{itemize}
\end{example}

We notice that the random $3-$uniform hypergraph $T(n,p)$ is the same as the random $3-$uniform hypergraph $R(n,p,1,3)$.

\section{Multi-action convergence for multi-linear operators}\label{SecMultiActConv}

In the previous section, we have seen how a hypergraph can be interpreted as a tensor and how there are various ways to interpret tensors as multi-linear operators. Therefore, we now want to generalize action convergence to general multi-linear operators.  

\begin{definition}\label{DefMultiPoper}
An $r-th$ order \emph{multi-$P-$operator}  is a multi-linear operator $A: L^{\infty}(\Omega) ^{r-1}\rightarrow$ $L^{1}(\Omega)$ such that the $\infty\rightarrow 1$ multi-linear operator norm 
\begin{equation*}
\|A\|_{\infty\rightarrow 1}:=\sup_{f^{(i)}\in L^{\infty}(\Omega), \, f^{(i)}\neq 0}\frac{\|A[f^{(1)},\ldots,f^{(r-1)}]\|_{ 1}}{\|f^{(1)}\|_{ \infty}\cdots\|f^{(r-1)}\|_{ \infty}}
\end{equation*}is finite. We will say that a multi-$P-$operator $A$ is \emph{acting} on the probability space $(\Omega,\mathcal{F}, \P)$ if the $L^1$ and $L^\infty$ spaces are defined on $(\Omega,\mathcal{F}, \P)$. We denote the set of all $r-$th order multi-$P-$operators with $\mathcal{B}_r$ and the set of all $r-$th order multi-$P-$operators acting on $(\Omega, \mathcal{F}, \P)$ with $\mathcal{B}_r(\Omega,\mathcal{F}, \P)$.
\end{definition}

We can relate multi-$P$ operators and tensors in multiple ways as in the following
\begin{example}
We can interpret the $s-$action of an $r-$th order symmetric tensor as a multi-$P-$operator 
$$\widetilde{T}:L^{\infty}([n]^s,Sym)^{r-1}\longrightarrow L^1([n]^s,Sym),$$
where $Sym$ is the symmetric $\sigma-$algebra on $[n]^s$ and we consider the uniform probability measure on $[n]^s$, i.e.\ $$\begin{aligned}
&\mathbb{P}(\{(j_1,\ldots,j_s) \text{ s.t.\ } (j_1,\ldots,j_s)=(i_{\sigma(1)},\ldots,i_{\sigma(s)}) \text{ where } \sigma \text{ is a permutation of $[s]$}\})\\
&=\frac{|\{(j_1,\ldots,j_s) \text{ s.t.\ } (j_1,\ldots,j_s)=(i_{\sigma(1)},\ldots,i_{\sigma(s)}) \text{ where } \sigma \in \mathcal{P}\}|}{n^s}
\end{aligned}$$
for all $i_1,\ldots, i_s \in [n]$. We just have to identify the set of $s-$th order symmetric tensors with $L^{\infty}([n]^s,Sym)\cong L^1([n]^s,Sym)$ in the canonical way.
\end{example}

{\color{black}

\begin{remark}
One also consider the non-symmetrized $s-$action as a multi-$P$ operator. The probability space to consider in that case is just $[n]^s$ with the discrete $\sigma-$algebra and the uniform probability measure on $[n]^s.$
\end{remark}
}

For functions $f^{(1)}_1,\ldots, f^{(r-1)}_1,\ldots, f^{(1)}_k,\ldots,f^{(r-1)}_k \in L^{\infty}(\Omega)$ we consider the $rk-$
dimensional random vector
$$
(f^{(1)}_1,\ldots, f^{(r-1)}_1,A[f^{(1)}_1,\ldots ,f^{(r-1)}_1],\ldots,f^{(1)}_k,\ldots,f^{(r-1)}_k,A[f^{(1)}_k,\ldots,f^{(r-1)}_k])
$$
for a multi-$P-$operator $A$ and we call the distribution of this random vector, 
\begin{equation}\label{EqConstructMultiAct}
\begin{aligned}
\mathcal{L}
(f^{(1)}_1,\ldots, f^{(r-1)}_1,A[f^{(1)}_1,\ldots f^{(r-1)}_1],\ldots,f^{(1)}_k,\ldots,f^{(r-1)}_k,A[f^{(1)}_k,\ldots,f^{(r-1)}_k]),
\end{aligned}
\end{equation}
 which is a probability measure in $\mathcal{P}(\mathbb{R}^{rk}),$ the measure generated by the multi-$P-$operator $A$ through the ordered sequence of functions $f^{(1)}_1,$$\ldots,$$ f^{(r-1)}_1,$$\ldots,  f^{(1)}_k,\ldots,f^{(r-1)}_k $$\in L^{\infty}(\Omega)$.
Sometimes, we will use the abbreviation 
$$
\begin{aligned}
 &\cD_A(f^{(1)}_1,\ldots, f^{(r-1)}_1\ldots,f^{(1)}_k,\ldots,f^{(r-1)}_k)\\ &=\mathcal{L}
(f^{(1)}_1,\ldots, f^{(r-1)}_1,A[f^{(1)}_1,\ldots, f^{(r-1)}_1],\ldots, f^{(1)}_k,\ldots,f^{(r-1)}_k,A[f^{(1)}_k,\ldots,f^{(r-1)}_k]).   
\end{aligned}
$$

We now define the set of measures generated by $A$. Similarly to the action convergence in the linear case, it is convenient to allow in our sets only measures generated by functions in $L_{[-1,1]}^{\infty}(\Omega)$, i.e. functions taking values between $-1$ and $+1$ almost everywhere. Therefore, we define the $k-$\emph{profile} of $A$,  $\cS_k(A)$, as the set of measures generated by $A$ by functions in  $L_{[-1,1]}^{\infty}(\Omega)$.
$$
\begin{aligned}
\cS_k(A)=\bigcup_{f^{(1)}_1,\ldots, f^{(r-1)}_1,\ldots, f^{(1)}_k,\ldots,f^{(r-1)}_k\in L_{[-1,1]}^{\infty}(\Omega)}&\{\mathcal{L}
(f^{(1)}_1,\ldots, f^{(r-1)}_1,A[f^{(1)}_1,\ldots f^{(r-1)}_1],\ldots,\\
&f^{(1)}_k,\ldots,f^{(r-1)}_k,A[f^{(1)}_k,\ldots,f^{(r-1)}_k])\}\\
\end{aligned}
$${
$$
\quad \quad=\bigcup_{f^{(1)}_1,\ldots, f^{(r-1)}_1,\ldots, f^{(1)}_k,\ldots,f^{(r-1)}_k\in L_{[-1,1]}^{\infty}(\Omega)}\cD_A(f^{(1)}_1,\ldots, f^{(r-1)}_1\ldots,f^{(1)}_k,\ldots,f^{(r-1)}_k).$$
\color{black} This is a set of measures. To compare two different sets of measures we will use the Hausdorff metric (Definition	\ref{DefHausdorffDist}) on sets of the space of probability measures $\mathcal{P}(\mathbb{R}^{rk})$ (equipped with the Levy-Prokhorov metric $d_{\mathcal{LP}}$ (Definition \ref{LevyProk})), that we will denote with $d_H.$ }

{\color{black}\begin{remark}
This is a generalization of the construction in Section \ref{SecActionConve}, see \eqref{EqDefKprofPoper}. We use a similar notation and terminology to underline the connection with action convergence \cite{backhausz2018action}. 
\end{remark}}

We are now ready to define the pseudo-metric we are interested in.
Consider two multi-$P-$operators 
$$
A: L^{\infty}(\Omega_1)^{r-1}\rightarrow L^1(\Omega_1)
$$and
$$
B: L^{\infty}(\Omega_2)^{r-1}\rightarrow L^1(\Omega_2).
$$
\begin{definition}[Metrization of action convergence] For the two $r-th$ order multi-$P$-operators $A, B$ the \emph{action convergence metric} is
$$
d_{M}(A, B):=\sum_{k=1}^{\infty} 2^{-k} d_{H}\left(\cS_{k}(A), \cS_{k}(B)\right).
$$
\end{definition}

{\color{black}

\begin{remark}
This is a generalization of the action convergence metric defined in Section \ref{SecActionConve}. For this reason we use the same notation here. However, this metric can be applied to multi-linear operators differently from the metric defined in Section \ref{SecActionConve}.    
\end{remark}

}

As the Hausdorff metric $d_H$ is bounded by 1, we have that also the action convergence distance is bounded by 1.

We will say that a sequence of $P$-operators $\left\{A_i \in \mathcal{B}_r\left(\Omega_i\right)\right\}_{i=1}^{\infty}$ is a Cauchy sequence if the sequence is Cauchy in $d_M$. \newline

We notice that a sequence $\left\{A_i \in \mathcal{B}_r\left(\Omega_i\right)\right\}_{i=1}^{\infty}$ is a Cauchy sequence  if and only if for every $k \in \mathbb{N}$ the sequence $\left\{\cS_k\left(A_i\right)\right\}_{i=1}^{\infty}$ is a Cauchy sequence in $d_H$.

\begin{remark} The completeness of $\left(\mathcal{P}\left(\mathbb{R}^k\right), d_{\mathcal{LP}}\right)$ implies that the induced Hausdorff topology is also complete \cite{RealAnalysisGordon}. Therefore, a sequence $\left\{A_i\right\}_{i=1}^{\infty}$ is a Cauchy sequence if and only if for every $k \in \mathbb{N}$ there is a closed set of measures $X_k$ such that $\lim _{i \rightarrow \infty} d_H\left(\cS_k\left(A_i\right), X_k\right)=0$.
 \end{remark}
 
The following lemma is an equivalent of Lemma 2.6 in \cite{backhausz2018action} for multi-$P-$operators and guarantees that a subsequence $\{{\cS_k\left(A_i\right)}\}^{\infty}_{i=1}$ converges in $d_H$ to a closed set of measures $X_k$ under a uniform bound assumption on the $\|\cdot\|_{\infty\rightarrow 1}$ norm.
 
\begin{lemma}
Let $\{A_i\}^{\infty}_{i=1}$ be a sequence of $r-$th order multi-$P-$operators with uniformly bounded $(\infty,1)-$norms. Then, it has a subsequence that is a Cauchy sequence.
\end{lemma}

This lemma follows directly from the same standard arguments that we summarize here for completeness.

For a probability measure $\mu$ on $\mathbb{R}^k$ let $\tau(\mu)\in [0,\infty]$ denote the maximal expectation of the marginals of $\mu$,
\begin{equation}\label{eqn:tau}\tau(\mu)=\max_{1\leq i\leq k} \int_{(x_1,x_2,\dots,x_k)\in\mathbb{R}^k}|x_i|~d\mu.\end{equation} For $c\in\mathbb{R}^+$ and $k\in\mathbb{N}$ let 
$$\mathcal{P}_c(\mathbb{R}^k):=\{\mu : \mu\in\mathcal{P}(\mathbb{R}^k) , \tau(\mu)\leq c\}.$$
Let furthermore $\mathcal{Q}_c(\mathbb{R}^k)$ denote the set of closed sets in the metric space $(\mathcal{P}_c(\mathbb{R}^k),d_{\mathcal{LP}})$. 

\begin{lemma} The metric spaces $(\mathcal{P}_c(\mathbb{R}^k),d_{\mathcal{LP}})$ and $(\mathcal{Q}_c(\mathbb{R}^k),d_H)$ are both compact and complete metric spaces.
\end{lemma}

\begin{proof} Markov's inequality gives uniform tightness in $\mathcal{P}_c(\mathbb{R}^k)$, which implies the compactness of $(\mathcal{P}_c(\mathbb{R}^k),d_{\mathcal{LP}})$ for Prokhorov's theorem. It is known that the set of closed subsets of a compact Polish space equipped with the Hausdorff metric is again compact.
\end{proof}

\begin{lemma} Let $A\in\mathcal{B}_r(\Omega)$ and let $c:=\max(\|A\|_{\infty\to1},1)$. Then for every $k\in\mathbb{N}$ the closure of $\cS_k(A)$ with respect to $d_{\mathcal{LP}}$ is in $\mathcal{Q}_c(\mathbb{R}^{rk})$.
\end{lemma}

\begin{proof}
Let $\{v^{(1)}_i,\ldots,v_i^{(r-1)}\}_{i=1}^k$ be a sequence of functions in $L^\infty_{[-1,1]}(\Omega)$. We have that $\|v^{(j)}_i\|_1\leq \|v^{(j)}_i\|_\infty\leq 1$ for every $j\in[r-1]$ and $\|A[v^{(1)}_i,\ldots,v^{(r-1)}_i ]\|_1\leq \|A\|_{\infty\to 1}$  holds for $1\leq i\leq k$. The result follows as the $1-$moments of the absolute values of the coordinates in $\tau$, \eqref{eqn:tau}, are given by $$\{\|v^{(j)}_i\|_1\}_{i=1}^k$$ for $j\in [r-1]$ and $$\{\|A[v^{(1)}_i,\ldots,v^{(r-1)}_i]\|_1\}_{i=1}^k.$$
\end{proof}

As in the linear case, for a sequence of multi-$P-$operators, we will not be interested only in the convergence of the sequences of $k-$profiles $\{\cS_k(A_i)\}^{\infty}_{i=1}$ but also in the existence of a multi-$P-$operator as limit object. This will actually be the convergence we are interested in.

\begin{definition}[Action convergence of multi-$P$-operators] 
We say that the sequence  $\left\{A_i \in \mathcal{B}_r\left(\Omega_i\right)\right\}_{i=1}^{\infty}$ is \emph{action convergent}  to the $r-$th order multi-$P-$operator $A\in\mathcal{B}_r\left(\Omega\right)$  if it is a Cauchy sequence and it is such that for every positive integer $k$ the $k-$profile $\cS_k(A)$ is the limit of the $k-$profiles sequence $\{\cS_k(A_i)\}^{\infty}_{i}$ in the Hausdorff metric $d_H.$ The multi-$P-$operator $A$ is the limit of the sequence  $\left\{A_i \right\}_{i=1}^{\infty}$.
\end{definition}
Additionally, we will say that a sequence of multi-$P-$operators $\left\{A_i \in \mathcal{B}_r\left(\Omega_i\right)\right\}_{i=1}^{\infty}$ is action convergent if there exists a limit multi-$P-$operator.

 \begin{remark}
 We will often use the following consequence of the definition of action convergence. For an action convergent sequence of operators $\left\{A_i\right\}_{i=1}^{\infty}$ to a multi-$P-$operator $A$ and for every $v^{(1)},\ldots, v^{(r-1)}\in L_{[-1,1]}^{\infty}(\Omega)$, there are elements $v^{(1)}_i,\ldots, v^{(r-1)}_i \in L_{[-1,1]}^{\infty}\left(\Omega_i\right)$ such that $$\mathcal{L}\left(v^{(1)}_i,\ldots, v^{(r-1)}_i,A_i[v^{(1)}_i,\ldots, v^{(r-1)}_i]\right)$$ weakly converges to $$\mathcal{L}(v^{(1)},\ldots, v^{(r-1)},A[v^{(1)},\ldots, v^{(r-1)}])$$ as $i$ goes to infinity.
 \end{remark}

We introduce now a multi-$P-$operator norm for $L^p$ spaces that is a natural generalization of the linear operator norm
\begin{definition}[Multi-linear operator norm]
For an $r-$th order multi-$P-$operator $A$ the multi-linear operator $(p_1,\ldots,p_{r-1},q)-$norm is 
\begin{equation*}
\|A\|_{p_1,\ldots,p_{r-1}\rightarrow q}:=\sup_{f^{(i)}\in L^{\infty}(\Omega), \, f^{(i)}\neq 0}\frac{\|A[f^{(1)},\ldots,f^{(r-1)}]\|_{q}}{\|f^{(1)}\|_{ p_1}\cdots\|f^{(r-1)}\|_{ p_r}}.
\end{equation*}

We denote the set of all $r-$th order multi-$P-$operators with finite $(p_1,\ldots,p_{r-1},q)-$
norm with $\mathcal{B}_{p_1,\ldots,p_{r-1},q}$ and the set of all $r-$th order multi-$P-$operators acting on $(\Omega,\mathcal{F}, \P)$ with finite $(p_1,\ldots,p_{r-1},q)-$norm with $\mathcal{B}_{p_1,\ldots,p_{r-1},q}(\Omega, \mathcal{F}, \P)$.
\end{definition}

\begin{remark}
With an abuse of notation, we can think of a multi-$P-$operator $A$ with bounded $(p_1,\ldots,p_{r-1},q)-$norm as a multi-linear bounded operator $$A:L^{p_1}(\Omega)\times \ldots \times L^{p_r}(\Omega)\rightarrow L^{q}(\Omega)$$ by Lemma \ref{LemmMultBoundExt}.
\end{remark}

The following theorem is the generalization of Theorem \ref{THMCompBack} (Theorem 2.9 in \cite{backhausz2018action}) to the multi-linear case and it states that sets of multi-$P-$operators with uniformly bounded $(p_1,\ldots,p_{r-1},q)-$norm with $p_1,\ldots,p_{r-1}\neq \infty$ are pre-compact in the action convergence metric.

\begin{theorem} \label{CompactnesMultiActConv} For $C>0,$ $p\in [1,\infty)$ and $q\in [1,\infty]$, let $\{A_i\}_{i=1}^\infty$ be a Cauchy sequence of $r-$th order multi-$P$-operators with uniformly bounded $\|\cdot\|_{p,\ldots,p\to q}$ norms. Then there is a multi-$P$-operator $A$ such that $\lim_{i\to\infty} d_M(A_i,A)=0$, and $\|A\|_{p,\ldots,p\to q}\leq\sup_{i\in\mathbb{N}}\|A_i\|_{p,\ldots,p\to q}\leq C$. Therefore, the sequence $\{A_i\}_{i=1}^\infty$ is action convergent. 
\end{theorem}

We give the technical proof of this theorem in the next section which is an adaptation of the proof of Theorem 2.9 in \cite{backhausz2018action} to the multi-linear case.

{\color{black}
\begin{remark}
Observe that having a uniform bound on the norm $\|\cdot\|_{p_1,\ldots,p_{r-1}\rightarrow q}$ for $p_1,\ldots,p_{r-1}\in [1,\infty)$ directly implies that we have a uniform bound on the norm $\|\cdot\|_{p,\ldots,p\to q}$ for $p=(\max_{i\in [r-1]}p_i)\in [1,\infty)$ as 
$$
  \|\cdot\|_{p_1,\ldots,p_{r-1}\rightarrow q }\leq \|\cdot\|_{p,\ldots,p\to q}.
$$
\end{remark}
}

\section{Construction of the limit object}\label{SecConstrLimObj}

For an $r-$th order multi-$P$-operator $A$ and $k\in\mathbb{N}$ let $cl(\cS_k(A))$ denote the closure of $\cS_k(A)$ in the space $(\mathcal{P}(\mathbb{R}^{rk}),d_{\mathcal{LP}})$.

This section is dedicated to showing Theorem \ref{CompactnesMultiActConv}. This technical proof is a generalization of the proof of Theorem 2.9 in \cite{backhausz2018action} to the multi-linear case (the proof is similar but we have to deal with multi-linear operators and a heavier notation). Let's consider $\left\{\left(\Omega_{i}, \mathcal{A}_{i}, \mu_{i}\right)\right\}_{i=1}^{\infty}$ a sequence of probability spaces and let's assume that $\left\{A_{i} \right\}_{i=1}^{\infty}$ is a Cauchy sequence of $P$-operators $A_{i} \in \mathcal{B}_{p_1,\ldots,p_{r-1}, q}\left(\Omega_{i}\right)$ with $\sup_i \left\|A_{i}\right\|_{p_1,\ldots, p_{r-1}\rightarrow q} \leq c$ for a fixed $c \in \mathbb{R}^{+}$. For every $k \in \mathbb{N}$, we can define
$$
X_{k}:=\lim _{i \rightarrow \infty} cl(\cS_{k}\left(A_{i}\right)).
$$
We aim to construct a multi-$P$-operator with $k$-profile that is the limit of the $k$-profiles of the operators in a given convergent sequence of operators for every fixed $k$, i.e.\ we will prove that there is a $P$-operator $A \in \mathcal{B}_{p_1,\ldots,p_{r-1}, q}(\Omega)$ for some probability space $(\Omega, \mathcal{A}, \mu)$ such that for every $k \in \mathbb{N}$ we have that
$$
\lim _{i \rightarrow \infty} cl(\cS_{k}\left(A_{i}\right))=cl(\cS_{k}(A)) .
$$
Before the technical proof, we describe the main idea. For every $k\in \N$ we consider the limit of the $k$-profiles $\cS_k(A_i)$ of the sequence of operators $A_i$, which is a set of measures, and we take a dense countable subset of this set. In this way, we have that each point in this dense subset can be approximated by elements in the $k$-profiles of the sequence of operators $A_{i}$. Moreover, every element in the $k$-profile of $A_{i}$ involves $r k$ measurable functions on $\Omega_{i}$ (in the terminology used before the measure is generated through those functions). In probabilistic language, these functions are random variables, since $\Omega_{i}$ is a probability space. Very roughly speaking, the main idea is to take, for every $k$, enough functions needed to generate enough measures (contained in the $k-$profiles of the operators $A_{i}$) to approximate a dense countable subset of the limiting $k-$profile. These are countably many functions for each $i$. By passing to a subsequence, we can assume that the joint distributions of these countably many functions (random variables) converge weakly and the limit is some probability measure on $\Omega:=\mathbb{R}^{\infty}$. 
Each coordinate function in the probability space on $\mathbb{R}^{\infty}$ corresponds to a function involved in a $k$-profile for some $k$. Since every measure in the $k$-profile comes from $(r-1)k$ functions and their $k$ images, we obtain some information on a possible limiting operator. More precisely, we obtain that certain coordinate functions are the images of some other coordinate functions under the action of the candidate limit multi-linear operator. However, it is not clear that it is possible to extend the obtained multi-linear operator to the full function space on $\Omega$ and so we need to refine the above idea. 

We now make the above idea rigorous. 
We need to work with enough functions to represent the function space of a whole $\sigma$-algebra to extend the candidate limit multi-linear operator for the entire function space on $\Omega$. To do this, we extend the above function systems by new functions obtained by some natural operations. In order to do this, we introduce an abstract algebraic formalism involving semigroups. The most challenging part of the proof is to show that, at the end of this construction, the limit operator is well-defined and has the desired properties.

For this construction, we will use the following algebraic notion.
\begin{definition}[Free semigroup with $r-$multi-operators]\label{FreeSemigroupDef}
Let $G$ and $L$ be sets. We denote by $F(G, L)$ the free semigroup with generator set $G$ and $r-$multioperator set $L$ (freely acting on $F(G, L))$. More precisely, we have that $F(G, L)$ is the smallest set of abstract words satisfying the following properties.
(1) $G \subseteq F(G, L)$.
(2) If $w_{1}, w_{2} \in F(G, L)$, then $w_{1} w_{2} \in F(G, L)$.
(3) If $w_1,\ldots, w_{r-1} \in F(G, L), l \in L$, then $l(w_1,\ldots,w_{r-1}) \in F(G, L)$. There is a unique length function $m: F(G, L) \rightarrow \mathbb{N}$ such that $m(g)=1$ for $g \in G$, $m\left(w_{1} w_{2}\right)=m\left(w_{1}\right)+m\left(w_{2}\right)$ and $m(l(w_1,\ldots,w_{r-1}))=\max_{s\in [r-1] }m(w_s)+1 .$
\end{definition} 

We give an example of a word in $F(G, L)$ with $L$ set of  $2-$multioperators: $$l_{3}\left(l_{1}\left(g_{1} ,g_{2}\right), l_{2}\left(g_{2} ,g_{2}g_{3}\right)\right)l_{3}\left(g_{1},g_{2}\right) ,$$ where $g_{1}, g_{2}, g_{3} \in G$ and $l_{1}, l_{2}, l_{3} \in L$. The length of this word is $\max\{\max\{1,1\}+1,\max\{1,1+1\}+1\}+1+\max\{1,1\}+1=6$. Note that if both $G$ and $L$ are countable sets, then also $F(G, L)$ is countable.\newline

In the first technical part of the proof, we construct a function system $\left\{v_{i, f} \in L^{\infty}\left(\Omega_{i}\right)\right\}, $
$i \in \mathbb{N}, f \in F$ for some countable index set $F$. Later, we will construct a probability measure $\kappa \in \mathcal{P}\left(\mathbb{R}^{F^{r-1} \times[r]}\right)$ and an operator $A \in \mathcal{B}_{p_1,\ldots,p_{r-1}, q}\left(\mathbb{R}^{F^{r-1} \times[r]}, \kappa\right)$ using this function system. In the end, we will show that $A$ is an appropriate limit object for the sequence $\left\{A_{i}\right\}_{i=1}^{\infty}$.\newline

Construction of a function system: First, we define $F$, the countable index set. For every $k \in \mathbb{N}$, let's consider $X_{k}^{\prime} \subseteq X_{k}$ a dense countable subset in the metric space $\left(X_{k}, d_{\mathcal{LP}}\right)$, which is separable. Let's define $G:=\bigcup_{k=1}^{\infty} X_{k}^{\prime} \times[k]\times [r-1]$, the generator set. Therefore, the index set $F$ will be the free semigroup generated by $G$ and a set of appropriate nonlinear $(r-1)-$multi-operators $L$. For any $y \in \mathbb{Q}$ and $z \in \mathbb{Q}^{+}$ let $h_{y, z}: \mathbb{R} \rightarrow [0,1]$ be the (bounded) continuous function defined by $h_{y, z}(x)=0$ for $x \notin(y-z, y+z)$ and $h_{y, z}(x)=1-|x-y| / z$ for $x \in$ $(y-z, y+z)$. Finally, for every $i \in \mathbb{N}, l \in L$ and $v_1,\ldots, v_{r-1} \in L^{\infty}\left(\Omega_{i}\right)$ we define $l(v_1,\ldots,v_{r-1}):=h_{y, z} \circ\left( A_{i}[v_1,\ldots,v_{r-1}]\right)$, where $l$ is indexed by the pair $(y, z) \in \mathbb{Q} \times \mathbb{Q}^{+}$. Observe that by definition, $\|l(v_1,\ldots,v_{r-1})\|_{\infty} \leq 1$. Being these functions indexed by $\mathbb{Q} \times \mathbb{Q}^{+}$, with an abuse of notation, we will denote $L=\mathbb{Q} \times \mathbb{Q}^{+}$. Therefore, we let $F:=F(G, L)$ be as in Definition \ref{FreeSemigroupDef} and, thus, $F$ is countable.
Furthermore, we define the functions $\left\{v_{i, g}\right\}_{i \in \mathbb{N}, g \in G}$. For every $i, k \in \mathbb{N}$, and $t \in X_{k}^{\prime}$ let $\left\{v_{i,(t, j,s)}\right\}_{j\in [k],s\in [r-1]}$ be random variables in $L_{[-1,1]}^{\infty}\left(\Omega_{i}\right)$ such that the joint distribution of
$$
\begin{aligned}
&(v_{i,(t, 1,1)},\ldots,v_{i,(t, 1,(r-1))},A_{i}[v_{i,(t, 1,1)},\ldots, v_{i,(t, 1,(r-1))}], v_{i,(t, 2,1)},\ldots,\\
&v_{i,(t, k,1)},\ldots,v_{i,(t, k,(r-1))},A_{i}[v_{i,(t, k,1)},\ldots, v_{i,(t, k,({r-1}))}])
\end{aligned}
$$%v_{i,(t, 2)}, %\ldots, v_{i,(t, k)}, v_{i,(t, 1)} A_{i}, v_{i,(t, 2)} A_{i}, \ldots, v_{i,(t, k)} A_{i}\right)
converges to $t$ as $i$ goes to $\infty$.

At this point, we will define the functions $\left\{v_{i, w}\right\}_{i \in \mathbb{N}, w \in F}$ recursively to the length of the words $m(w)$. The functions have been constructed above for words of length $m(w)=1$. Assume we have already constructed all the functions $v_{i, w}$ with $m(w) \leq j$ for some $j \in \mathbb{N}$. Consider a $w \in F$ such that $m(w)=j+1$. If $w=w_{1} w_{2}$ for some $w_{1}, w_{2} \in F$, then set $v_{i, w}:=v_{i, w_{1}} v_{i, w_{2}}$. If $w=l\left(w_{1},w_2,\ldots,w_{(r-1)}\right)$, then set $v_{i, w}:=l\left(v_{i, w_{1}},v_{i, w_{2}},\ldots,v_{i, w_{r-1}}\right).$

Construction of the probability space: Let $\xi_{i}: \Omega_{i} \rightarrow \mathbb{R}^{F^{(r-1)} \times[r]}$ be the map such that for $f_1,\ldots,f_{(r-1)} \in F, e \in[r]$, and $\omega_{i} \in \Omega_{i}$ the $(f_1,\ldots,f_{(r-1)}, e)$ coordinate of $\xi_{i}\left(\omega_{i}\right)$ is equal to $$\left(A_{i}^{e}[v_{i, f_1},\ldots, v_{i, f_{(r-1)} }]\right)\left(\omega_{i}\right),$$ where $A_{i}^{s}$ for $s\in [r-1]$ is defined to be the projection on the $s-$th variable and $A_i^r=A_i$. For the random variable $\xi_{i}$ we denote its distribution with $\kappa_{i} \in \mathcal{P}(\mathbb{R}^{F^{(r-1)} \times [r]})$, i.e.\ $\kappa_{i}$ is the joint distribution of the functions $\left\{v_{i, f_1}\right\}_{f_{1} \in F},\ldots,\left\{v_{i, f_{r-1}}\right\}_{f_{r-1} \in F}$ and $\left\{A_{i}[v_{i, f_1},\ldots ,v_{i, f_{r-1}} ]\right\}_{f_1,\ldots,f_{(r-1)} \in F}$. Since $\tau\left(\kappa_{i}\right) \leq c$ holds (we recall the definition of $\tau$, equation \eqref{eqn:tau}), there exists a strictly increasing sequence $\left\{n_{i}\right\}_{i=1}^{\infty}$ in $\mathbb{N}$ and a probability measure $\kappa\in \mathcal{P}\left(\mathbb{R}^{F^{(r-1)} \times [r]}\right)$ such that $\kappa_{n_{i}}$ is weakly convergent to $\kappa$ as $i$ goes to infinity. We will define $\Omega:=\mathbb{R}^{F^{(r-1)} \times[r]}$ and consider $\Omega$ as a topological space, equipped with the product topology. Therefore, we constructed the probability space $(\Omega,\mathcal{A},\kappa)$, where the $\sigma-$algebra $\mathcal{A}$ is its Borel $\sigma$-algebra and $\kappa$ the probability measure obtained as weak limit of the sequence $\kappa_{n_{i}}$. We remark that $\kappa$ is a probability measure, as it is the weak limit of probability distributions.

Construction of the operator: We now define an operator $A \in \mathcal{B}_{p_1,\ldots,p_{(r-1)}, q}(\Omega)$ with the probability space $\Omega$ defined above. For $(f_1,\ldots,f_{r-1}, e) \in F^{(r-1)} \times [r]$ we denote with $\pi_{(f_1,\ldots,f_{(r-1)}, e)}: \mathbb{R}^{F^{(r-1)} \times [r]} \rightarrow \mathbb{R}$ the projection to the $(f_1,\ldots,f_{(r-1)}, e)$ coordinate. Observe that \begin{equation}\label{eqn:EqoperatorLin}
\begin{aligned}
\pi_{(f_1,\ldots,f_{(r-1)}, e)}& \circ \xi_{i}\\
&= A_{i}^{e}[v_{i, f_1},\ldots ,v_{i, f_{r-1}} ] \quad(i \in \mathbb{N},(f_1,\ldots,f_{(r-1)}, e) \in F^{(r-1)}\times[r])  .
\end{aligned}
\end{equation}
Additionally, by the definition of $\kappa$, we also notice that $\pi_{(f_1,\ldots,f_{(r-1)}, e)} \in L_{[-1,1]}^{\infty}(\Omega)$ for $f_1,\ldots,f_{(r-1)} \in F$ and $e\in [r-1]$. We want now to prove that there exists a unique $(p_1,\ldots,p_{(r-1)}, q)$-bounded $(r-1)-$th order multi-$P-$operator $A$ from $L^{\infty}(\Omega)\times \ldots \times L^{\infty}(\Omega)$ to $L^{1}(\Omega)$ with $\|A\|_{p_1,\ldots,p_{(r-1)} \rightarrow q} \leq c$ such that $A[\pi_{(f_1,\ldots, f_{r-1}, 1)},\ldots,\pi_{(f_1,\ldots, f_{r-1}, r-1)} ]$
$=\pi_{(f_1,\ldots, f_{r-1}, r)}$  holds for every $f_1,\ldots, f_{r-1} \in F$.

\begin{lemma}\label{LemmaPropertiesCoordinateFunct}
 For the coordinate functions on $\mathbb{R}^{F^{(r-1)} \times[r]}$ the following properties hold:
 \begin{enumerate}
\item  If $e\in [r-1]$ and $f_{1}, f_{2} \in F$, then $\pi_{\left(\cdot,\ldots,\cdot ,f_{1} f_{2},\cdot,\ldots,\cdot,  e\right)}=\pi_{\left(\cdot,\ldots,\cdot ,f_{1},\cdot,\ldots,\cdot , e\right)} \pi_{(\cdot,\ldots,\cdot ,f_{2},\cdot,\ldots,\cdot ,e)}$ holds in $L^{\infty}(\Omega)$.
\item  If $f_1\ldots, f_{r-1}\in F$ and $l=(y, z)\in L,$ then $\pi_{(\cdot,\ldots,\cdot ,l(f_1,\ldots,f_{r-1}),\cdot,\ldots,\cdot, e)}=h_{y, z} \circ \pi_{(f_1,\ldots,f_{(r-1)}, r)}$ holds in $L^{\infty}(\Omega)$.
\item  If $a^{(1)}_{s}, a^{(2)}_{s}, \ldots, a_{s}^{(d_s)} \in F, \lambda_s^{(1)}, \lambda_s^{(2)}, \ldots, \lambda_s^{(d_s)} \in \mathbb{R}$, for every $s\in [r-1]$ then
$$\begin{aligned}
&\left\|\sum_{j_1,\ldots j_{(r-1)}=1}^{d_1,\ldots,d_{(r-1)}} \lambda_1^{(j_1)} \lambda_2^{(j_2)}\ldots \lambda_{(r-1)}^{(j_{(r-1)})} \pi_{\left(a^{(j_1)}_{1}, a^{(j_2)}_{2}, \ldots, a_{(r-1)}^{(j_s)}, r\right)}\right\|_{q} \\
& \leq c\left\|\sum_{j_1=1}^{d_1} \lambda_1^{(j_1)} \pi_{\left(a^{(j_1)}_{1},\cdot, \ldots, \cdot, 1\right)}\right\|_{p_1}\ldots\left\|\sum_{j_{r-1}=1}^{d_{(r-1)}} \lambda_{(r-1)}^{(j_{(r-1)})} \pi_{\left(\cdot, \ldots, \cdot,a^{(j_{r-1})}_{1}, r-1\right)}\right\|_{p_{r-1}}.
\end{aligned}
$$
\item  For all $e\in [r-1]$, the linear span of the functions $\left\{\pi_{(\cdot,\ldots , \cdot,f,\cdot, \ldots,\cdot, e)}\right\}_{f \in F}$ is dense in the space $L^{p_e}(\Omega)$.
\item  Assume that $k \in \mathbb{N}$ and $t \in X_{k}^{\prime}$. Then $(t, j,s) \in G \subset F$ holds for $1 \leq j \leq k, $ $s\in[r-1]$ and we have
\begin{equation*}\begin{aligned}
& \mathcal{L}\left(\pi_{((t, 1,1),\ldots, 1)}, \pi_{((t, 2,1),\ldots, 1)}, \ldots, \pi_{((t, k,1),\ldots,1)}, \pi_{(\cdot,(t, 1,2),\ldots, 2)}, \pi_{(\cdot, (t, 2,2),\ldots, 2)}, \ldots, \right. \\ & \left.\pi_{(\cdot,(t, k,2),\ldots, 2)},
 \ldots, \pi_{(\ldots,(t, k,r-1), r-1)},\ldots,\pi_{((t, k,1),\ldots,(t, k,r-1), r)}\right)=t.
 \end{aligned}
\end{equation*}

 \end{enumerate}
\end{lemma}
\begin{remark} When functions on $\Omega$ are treated as functions in $L^r(\Omega)$ for some $r\in [1,\infty]$, they are identified if they differ on a set of measure zero. This standard identification of functions allows the correspondence between different coordinate functions. For example let us consider the uniform measure $\mu$ on $\{(x,x):x\in [0,1]\}$ which is a Borel measure on $\mathbb{R}^2$. The $x$-coordinate function $(x,y)\mapsto x$ and the $y$-coordinate function $(x,y)\mapsto y$ coincide in the space $L^r(\mathbb{R}^2,\mu)$, as they agree on the support of $\mu$. We will heavily exploit this fact in the rest of our proof.
\end{remark}

For the proof of Lemma \ref{LemmaPropertiesCoordinateFunct} we will need the following two lemmas.
\begin{lemma}[Lemma 4.3 in \cite{backhausz2018action}]
 Let $r \in[1, \infty)$. For every $v \in L^{r}(\Omega)$ we have that
$$
\lim _{n \rightarrow \infty}\left\|v-\sum_{j=-n^{2}}^{n^{2}}(j / n) h_{j / n, 1 / n} \circ v\right\|_{r}=0 .
$$\end{lemma}

The following lemma, which is easy to show, see Theorem 22.4 in the lecture notes \cite{driver2004analysis}, will be needed in the following.

\begin{lemma}
Let $r \in[1, \infty)$. Let $\left\{v_{i} \in L^{\infty}(\Omega)\right\}_{i \in I}$ be a system of functions for some countable index set $I$ such that for every $a, b \in I$ there is $c \in I$ with $v_{a} v_{b}=v_{c}$. Let $\mathcal{A}_{0}$ be the $\sigma$-algebra generated by the functions $\left\{v_{i}\right\}_{i \in I}$. Suppose that the constant 1 function on $\Omega$ can be approximated by a uniformly bounded family of finite linear combinations of $\left\{v_{i}\right\}_{i \in I}$. Then the $L^{r}$-closure of the linear span of $\left\{v_{i} \in L^{\infty}(\Omega)\right\}_{i \in I}$ is $L^{r}\left(\Omega, \mathcal{A}_{0}, \kappa\right)$.
\end{lemma} 

Finally, we come back to the proof of Lemma \ref{LemmaPropertiesCoordinateFunct}.

\proof The first statement of the lemma is shown as follows. By the construction of the function system, for every $i \in \mathbb{N}$ and $f_{1}, f_{2} \in F$, it holds that $v_{i, f_{1} f_{2}}=v_{i, f_{1}} v_{i, f_{2}}$. Therefore, by equation \eqref{eqn:EqoperatorLin} and the continuity of $\pi$, it follows that each $\kappa_{i}$ is supported on the closed set
$$
\bigcap_{e\in [r-1] }\left\{\omega: \omega \in \mathbb{R}^{F^{r-1} \times[r]}, \pi_{\left(\ldots,f_{1} f_{2},\ldots, e\right)}(\omega)=\pi_{\left(\ldots,f_{1},\ldots, e\right)}(\omega) \pi_{\left(\ldots,f_{2},\ldots, e\right)}(\omega)\right\}.
$$
Therefore, $\kappa$ is also supported inside this set and hence the equality $\pi_{\left(\ldots,f_{1} f_{2},\ldots, e\right)}=\pi_{\left(\ldots,f_{1},\ldots, e\right)} \pi_{\left(\ldots,f_{2},\ldots, e\right)}$ holds $\kappa$-almost everywhere for every $e\in [r-1]$.

The second statement is proven along the same lines as the first one. Again, by the construction of the function system, it follows that for every $i \in \mathbb{N}$ and $f_1\ldots, f_{r-1} \in F, l=(y, z) \in L$ we have $v_{i,\left.l(f_1,\ldots,f_{r-1}\right)}=l\left(v_{i, f_1},\ldots,v_{i, f_{r-1}}\right)=h_{y, z} \circ\left( A_{i}[v_{i,f_1},\ldots, v_{i,f_{r-1}}]\right)$. Thus, by the definition of $\kappa_{i}$, equation \eqref{eqn:EqoperatorLin} and the continuity of $\pi$ we obtain that $\kappa_{i}$ is supported inside the closed set
$$
\bigcap_{e\in [r-1] } \left\{\omega: \omega \in\mathbb{R}^{F^{r-1} \times[r]}, \pi_{(\ldots, l(f_1,\ldots,f_{r-1}),\ldots, e)}(\omega)=h_{y, z}\left(\pi_{(f_1,\ldots,f_{r-1}, r)}(\omega)\right)\right\}
$$
for every $i \in \mathbb{N}$. Therefore, for every $e\in [r-1]$, the equality $\pi_{(\ldots,l(f_1,\ldots, f_{r-1}),\ldots, e)}=h_{y, z} \circ \pi_{(f_1,\ldots,f_{r-1}, r)}$ holds $\kappa$-almost everywhere.

To show the third claim, we recall that $\left\|A_{i}\right\|_{p_1,\ldots, p_{r-1} \rightarrow q} \leq c$ holds for every $i \in \mathbb{N}$ and thus
$$\begin{aligned}
\left\|\sum_{j_1,\ldots, j_{(r-1)}=1}^{d_1,\ldots,d_{(r-1)}} \lambda_1^{(j_1)}\ldots\lambda_{(r-1)}^{(j_{(r-1)})} \right.& \left. A_{i}[v_{i, a_{j_1}},\ldots, v_{i, a_{j_{(r-1)}}}]\right\|_{q} \\
&\leq c\left\|\sum_{j_1=1}^{d_1} \lambda_1^{(j_1)} v_{i, a_{j_1}}\right\|_{p_1}\ldots\left\|\sum_{j_{(r-1)}=1}^{d_{(r-1)}} \lambda_{(r-1)}^{(j_{(r-1)})} v_{i, a_{j_{r-1}}}\right\|_{p_{r-1}}.
\end{aligned}
$$The sums in the factors on the right-hand side are functions in $L^{\infty}\left(\Omega_{i}\right)$ whose values for the respective $e\in [r-1]$ are in the compact intervals $[-\lambda^{(e)}, \lambda^{(e)}]$ for $\lambda^{(e)}:=\sum_{j_e=1}^{d_e}\left|\lambda^{(e)}_{j_e}\right|$, therefore, we obtain that $\sum_{j_e=1}^{d_e} \lambda^{(j_e)}_{e} \pi_{\left(\ldots,a_{j_e},\ldots , e\right)}$ is a bounded, continuous function on the support of $\kappa$. Therefore, using that  $\kappa_{i} $ converges to $\kappa$ weakly and equation \eqref{eqn:EqoperatorLin} again (in particular, integrating the $p-$th power of the absolute values with respect to $\kappa_{i}$), we obtain that
$$
\lim _{i \rightarrow \infty}\left\|\sum_{j_e=1}^{d_e} \lambda_e^{(j_e)} v_{i, a_{j_e}}\right\|_{p_e}=\left\|\sum_{j_e=1}^{d_e} \lambda_e^{(j_e)} \pi_{\left(\ldots,a_{j_e},\ldots, e\right)}\right\|_{p_e}.
$$
On the other side, weak convergence implies the following inequality:
$$
\begin{aligned}
\left\|\sum_{j_1,\ldots,j_{(r-1)}=1}^{d_1,\ldots,d_{(r-1)}} \right.&\lambda_1^{(j_1)} \ldots \lambda_{(r-1)}^{(j_{(r-1)})}\pi_{(a_1,\ldots,a_{(r-1)}, r)}\Biggr\|_{q}\\
&\leq \limsup _{i \rightarrow \infty}\left\|\sum_{j_1,\ldots,j_{(r-1)}=1}^{d_1,\ldots,d_{(r-1)}}\lambda_1^{(j_1)} \ldots \lambda_{(r-1)}^{(j_{(r-1)})}A_{i}[v_{i, a_{j_1}},\ldots, v_{i, a_{j_{(r-1)}}}]\right\|_{q} .
\end{aligned}$$ as $\left|\sum_{j_1,\ldots,j_{(r-1)}=1}^{d_1,\ldots,d_{(r-1)}} \lambda_1^{(j_1)} \ldots \lambda_{(r-1)}^{(j_{(r-1)})}\pi_{(a_1,\ldots,a_{(r-1)}, r)}\right|^{q}$ is a continuous non-negative function. Therefore, putting those inequalities together we obtain the third statement.

To prove the fourth statement, let $\mathcal{H}^{(e)}_{s}$ be the $L^{s}$-closure of the linear span of the function system $\left\{\pi_{(f_1,\ldots,f_{(r-1)}, e)}\right\}_{f_1,\ldots, f_{(r-1)} \in F}$ for $e\in [r-1]$  and $s \in[1, \infty)$. 

First of all we notice that
$$
\pi(f,\ldots,1) =\pi(\ldots,f,\ldots,e)=\pi(\ldots,f,r-1)
$$for all $f\in F$ and $e\in [r-1]$ and, therefore, $\mathcal{H}^{(1)}_{s}=\ldots=\mathcal{H}^{(e)}_{s}=\ldots=\mathcal{H}^{(r-1)}_{s}$. From now on we will write $\mathcal{H}_s=\mathcal{H}^{(e)}_s$ as it does not depend on $e$.

Now we prove that $\pi_{(f_1,\ldots,f_{(r-1)}, r)} \in \mathcal{H}_{q}$ holds for every $f_1,\ldots, f_{(r-1)} \in F$. From the second statement of the lemma, it follows that the following equality holds
 \begin{equation}\label{eqn:stepCompEqual}
\sum_{j=-n^{2}}^{n^{2}}(j / n) h_{j / n, 1 / n} \circ \pi_{(f_1,\ldots, f_{(r-1)}, r)}=\sum_{j=-n^{2}}^{n^{2}}(j / n) \pi_{(\ldots,l_j(f_1,\ldots,f_{(r-1)}), e)},
\end{equation}
where $l_{j}$ is represented by the pair $(j / n, 1 / n)$ for $-n^{2} \leq j \leq n^{2}$. 

Thus, we notice that the left-hand side is in $\mathcal{H}_{q}$ as the right-hand side of \eqref{eqn:stepCompEqual} obviously is in $\mathcal{H}_{q}$. Moreover, $\pi_{(f_1, \ldots,f_{(r-1)},r)} \in L^{q}(\Omega)$ by the third statement of the lemma. Hence, by Lemma \ref{LemmaPropertiesCoordinateFunct}, we have that,
the left-hand side of \eqref{eqn:stepCompEqual} converges to $\pi_{(f_1,\ldots,f_{(r-1)}, r)}$ in $L^{q}(\Omega)$, as $n$ goes to $\infty$, and hence $\pi_{(f_1,\ldots,f_{r-1}, r)} \in \mathcal{H}_{q}$.

For fixed $e\in [r-1]$, let $\mathcal{A}_{0}$ denote the $\sigma$-algebra generated by the functions $\left\{\pi_{(f_1,\ldots,f_{r-1}, e)}\right\}_{f_1,\ldots,f_{r-1} \in F}$. Observe that already in $X_{1}^{\prime}$ the constant function $1$ can be approximated on $\Omega$. Therefore, we obtain by the first statement in this lemma and Lemma \ref{LemmaPropertiesCoordinateFunct} that $\mathcal{H}_{r}=L^{r}\left(\Omega, \mathcal{A}_{0}, \kappa\right)$ holds for every $e\in [r-1]$ and $r \in$ $[1, \infty)$. Thus, we obtained that for every $f_1,\ldots, f_{r-1}\in F$ the equality $\pi_{(f_1,\ldots,f_{r-1}, r)} \in \mathcal{H}_{q}=L^{q}\left(\Omega, \mathcal{A}_{0}, \kappa\right)$ holds and, hence, all coordinate functions on $\mathbb{R}^{F^{r-1} \times[r]}$ are measurable in $\mathcal{A}_{0}$. This finally proves that $\mathcal{H}_{r}=L^{r}\left(\Omega, \mathcal{A}_{0}, \kappa\right)=L^{r}(\Omega, \mathcal{A}, \kappa)=L^{r}(\Omega)$ holds for every $r \in[1, \infty)$.

From the definition of the functions $\left\{v_{i,(t, j,e)}\right\}_{i \in \mathbb{N}, j \in[k]}$ and the definition of the probability measure $\kappa$, we directly obtain the last statement of the lemma.
\endproof

We will need also the following lemma to prove the existence of the multi-$P-$operator. This is the multi-linear version of a classical result about the extension of linear bounded operators defined on a dense set.

\begin{lemma}\label{LemmMultBoundExt}
Let $V_1,\ldots,V_r$ and $U$ be Banach spaces and $W_1,\ldots, W_r$ where, for every $i\in [r]$, $W_i$ is a dense subspace of $V_i$.
For a multi-linear bounded operator
$$
  T_0:W_1\times \ldots \times W_r \longrightarrow U     
$$
$$
(x_1,\ldots,x_r)\mapsto T[x_1,\ldots,x_r]
$$
there exists a unique multi-linear bounded operator 

$$
  T:V_1\times \ldots \times V_r \longrightarrow U   
$$

and 
$$
\|T_0\|=\|T\|.
$$
\end{lemma}
\proof
For every $(x_1,\ldots,x_r)\in V_1\times \ldots \times V_r$ we define
$$T[x_1,\ldots,x_r ]=\lim_{n\rightarrow \infty}T[x_{1,n},\ldots,x_{r,n}]$$
where $(x_{1,n},\ldots,x_{r,n})\rightarrow (x_1,\ldots,x_r)$ as $n\rightarrow \infty$ where $(x_{1,n},\ldots,x_{r,n})\in W_1\times \ldots \times W_r$ for every $n\in \N$ and the convergence is in the natural norm on $V_1\times \ldots \times V_r$. We show that this definition is independent of the sequence we choose. We consider two sequences

$$
(x_{1,n},\ldots,x_{r,n})\rightarrow (x_1,\ldots,x_r)
$$
$$
(y_{1,n},\ldots,y_{r,n})\rightarrow (x_1,\ldots,x_r)
$$

$$
\begin{aligned}
&\|T[x_{1,n},\ldots,x_{r,n}]-T[y_{1,n},\ldots,y_{r,n}]\|\\
& \leq\|T[x_{1,n},\ldots,x_{r,n}]-T[y_{1,n},y_{r-1,n},\ldots,x_{r,n}]+\ldots\\
& \hspace{2.0 cm}+ T[x_{1,n},y_{2,n},\ldots,y_{r,n}]-T[y_{1,n},\ldots,y_{r,n}]\|\\
&\leq
C \sum^r_{i=1}\left(\prod^{i-1}_{j=1}\|x_{j,n}\|\right)\|x_{i,n}-y_{i,n}\|\left(\prod^{r}_{j=i+1}\|y_{j,n}\|\right) \\
&\leq
K\sum^r_{i=1}\|x_{i,n}-y_{i,n}\|\rightarrow 0
\end{aligned}
$$
as $n\rightarrow 0$. Moreover,
$$
\begin{aligned}
&\|T_0\|=\sup_{x_1\in W_1,\ldots, x_r\in W_r, \
x_1,\ldots,x_r\neq 0}\frac{\|T[x_1,\ldots,x_r]\|}{\|x_1\|\ldots\|x_r\|}\\
&=\sup_{x_1\in V_1,\ldots, x_r\in V_r, \
x_1,\ldots,x_r\neq 0}\frac{\|T[x_1,\ldots,x_r]\|}{\|x_1\|\ldots\|x_r\|}=\|T\|&
\end{aligned}
$$

as the sets $W_i$ are dense in $V_i$ and therefore $W_1\times \ldots \times W_r$ is dense in $V_1\times \ldots \times V_r$.

\endproof

We now finally define the limit operator $A \in \mathcal{B}_{p_1,\ldots,p_{r-1}, q}(\Omega)$. For $f_1,\ldots,f_{r-1} \in F$, let 
$$ A[\pi_{(f_1,\ldots, 1)},\ldots,\pi_{(\ldots,f_e,\ldots, e)},\ldots, \pi_{(\ldots,f_{r-1}, r-1)}]=\pi_{(f_1,\ldots,f_{r-1},r)}.$$This defines a multi-linear operator on the linear span of $\left\{\pi_{(f_1,\ldots,f_{r-1}, e)}\right\}_{f_1,\ldots,f_{r-1} \in F}$. This operator is bounded by the third statement of Lemma \ref{LemmaPropertiesCoordinateFunct}. 
Thus, there exists a unique continuous multi-linear extension on its $L^{p_1}\times \ldots \times L^{p_{r-1}}$-closure. 
In fact, by the fourth statement of Lemma \ref{LemmaPropertiesCoordinateFunct} and Lemma \ref{LemmMultBoundExt}, we get that there is a unique operator $A \in \mathcal{B}_{p_1,\ldots,p_{r-1}, q}(\Omega)$ with $\|A\|_{p_1,\ldots,p_{r-1} \rightarrow q} \leq c$ such that $ A[\pi_{(f_1,\ldots, 1)},\ldots,\pi_{(\ldots,f_{r-1} ,r-1)}]=\pi_{(f_1,\ldots,f_{r-1}, r)}$ holds for every $f_1,\ldots,f_{r-1} \in F$.\newline
Last part of the proof: From the last statement of Lemma \ref{LemmaPropertiesCoordinateFunct} together with the equality $$A[\pi_{((t, j,1),\ldots, 1)},\ldots, \pi_{(\ldots, (t, j,r-1), r-1)}] =\pi_{((t, j, 1),\ldots,(t, j, r-1), r)}$$we obtain that for every $k \in \mathbb{N}$ and $t \in X_{k}^{\prime}$ it holds $t \in \cS_{k}(A)$. Hence, for every $k \in \mathbb{N}$ we directly observe that $X_{k} \subseteq cl(\cS_{k}(A))$. We now want to prove that $X_{k}=cl(\cS_{k}(A))$ for every $k \in \mathbb{N}$ and thus we still need to show the converse inclusion $cl(\cS_{k}(A) )\subseteq X_{k}$. Let $k \in \mathbb{N}$ and let $v_{1,1}, v_{2,1}, \ldots, v_{k,1}, \ldots ,v_{1,r-1}, v_{2,r-1}, \ldots, v_{k,r-1} \in L_{[-1,1]}^{\infty}(\Omega)$. Hence, we aim to prove that $$\alpha:=\mathcal{D}_{A}\left(\left\{v_{j,s}\right\}_{j\in[k],s\in [r-1]}\right)\in X_{k}.$$For $\varepsilon>0$ arbitrary, it follows by the fourth statement of Lemma \ref{LemmaPropertiesCoordinateFunct} that for some large enough natural number $m$ there are elements $f_{1,s}, f_{2,s}, \ldots, f_{m,s} \in F$ and real numbers $\left\{\lambda_{a, j,s}\right\}_{a \in[m], j \in[k] , s\in [r-1]}$ such that for every $j \in[k]$ we have $\left\|w_{j,s}-v_{j,s}\right\|_{p_s} \leq \varepsilon$, where $w_{j,s}:=$ $\sum_{a=1}^{m}$\ $ \lambda_{a, j,s}$  $ \pi_{\left(\ldots, f_{a,s},\ldots, s\right) }$ for $j \in[k], s\in [r-1]$.

We recall that only vectors with $\infty-$norm bounded by $1$ are admitted in the profiles. For this reason, we will need to use a truncation function $\tilde{h}$. Let $\tilde{h}: \mathbb{R} \rightarrow [-1,1]$ be the continuous function with $\tilde{h}(x)=x$ for $x\in[-1,1], \tilde{h}(x)=-1$ for $x\in(-\infty,-1]$ and $\tilde{h}(x)=1$ for $x \in[1, \infty)$. We notice that $\left|w_{j,s}(\omega)-v_{j,s}(\omega)\right| \geq\left|\tilde{h} \circ w_{j,s}(\omega)-v_{j,s}(\omega)\right|$ holds almost everywhere as $\left\|v_{j,s}\right\|_{\infty} \leq 1$ for every $j \in[k]$ and $s\in [r-1].$ By $\left\|w_{j,s}-v_{j,s}\right\|_{p_s} \leq \varepsilon$, we observe that $\left\|\tilde{h} \circ w_{j,s}-v_{j,s}\right\|_{p_s} \leq \varepsilon$ for $j \in [k]$. Therefore, using the triangle inequality we obtain
\begin{equation}\label{eqn:EqIneqLp}
\left\|\tilde{h} \circ w_{j,s}-w_{j,s}\right\|_{p_s} \leq\left\|\tilde{h} \circ w_{j,s}-v_{j,s}\right\|_{p_s}+\left\|v_{j,s}-w_{j,s}\right\|_{p_s} \leq 2 \varepsilon\end{equation}
for $j \in[k],s\in[r-1]$.
For $i \in \mathbb{N}, \ e\in [r-1]$ and $j \in[k]$ let $z_{i, j,s}:=\sum_{a=1}^{m} \lambda_{a, j,s} v_{i, f_{a,j,s}}$ and let $$\beta_{i}:=\mathcal{D}_{A_{i}}\left(\left\{z_{i, j,s}\right\}_{j\in [k], s\in [r-1]}\right).$$By the properties of convergence in distribution of random vectors (linear combinations of entries converge in distribution to the same linear combination of the entries of the limit random vector) and the definition of $\kappa$, it follows that
$$
\beta:=\lim _{i \rightarrow \infty} \beta_{n_{i}}=\mathcal{D}_{A}\left(\left\{w_{j,s}\right\}_{j\in [k],s\in [r-1]}\right)
$$
holds in $d_{\mathcal{LP}}$. Moreover, we have $$\begin{aligned}
&\left\| A[v_{j,1},\ldots, v_{j,r-1}]- A[w_{j,1}, \ldots , w_{j,r-1}]\right\|_{1} \\
&\leq\left\| A[v_{j,1},\ldots, v_{j,r-1}]- A[w_{j,1}, \ldots , w_{j,r-1}]\right\|_{q}  \\
&\leq \sum^{r-1}_{e=1}c \left(\prod^{e-1}_{s=1}\| v_{j,s}\|_{p_s}\right)\| v_{j,e}-w_{j,e} \|_{p_e}\left(\prod^{r-1}_{s=e+1}\| w_{j,s}\|_{p_s}\right)\\
%&\leq ........c (r-1)\max_{s\in [r-1]}\{\|v_{j,s}\|_{p_s}\}\max_{s\in [r-1]}\{\|w_{j,s}\|_{p_s}\}\varepsilon............
%\\
& \leq c (r-1)\max_{s\in [r-1]}\ \{\|v_{j,s}\|_{p_s}+\varepsilon\}^{r-2}\varepsilon\\
& \leq c (r-1)\{1+\varepsilon\}^{r-2}\varepsilon\leq C\ \varepsilon\end{aligned} $$
since  $w_{j,s} \in L^{\infty}(\Omega)$ and where the second last inequality follows from $\|v_{j,s}\|_{\infty}\leq 1$.
From Lemma \ref{coupdist2} we have that $d_{\mathcal{LP}}(\alpha, \beta) \leq(r k)^{3 / 4}\left(C^{\prime} \varepsilon\right)^{1 / 2}$, where $C^{\prime}:=$ $\max (C, 1)$.
Let
$$
\beta_{i}^{\prime}:=\mathcal{D}_{A_{i}}\left(\left\{\tilde{h} \circ z_{i, j,s}\right\}_{j\in [k],s\in [r-1]}\right).
$$Observe that the function $$f:\R\longrightarrow \R$$ $$f(x)=\tilde{h}(x)-x$$ is continuous. 
Moreover, the functions $z_{i,j,s}$ all take values in the compact interval $[-m \tilde\lambda,m\tilde\lambda]$ where $\tilde \lambda= \max_{a\in[m],j\in[k] ,s\in [r-1]}{|\lambda_{a,j,s}|}$. Therefore, it follows that 
$$
\|f(z_{i,j,s})\|_{p_s}\rightarrow\|f(w_{j,s})\|_{p_s}\leq 2 \varepsilon
$$
for $i\rightarrow \infty$ as $f$ is continuous, $z_{i,j,s}$ converge in distribution to $w_{j,s}$, $z_{i,j,s}$ are uniformly bounded and the inequality in \eqref{eqn:EqIneqLp}.
Hence, if $i$ is large enough, then $\left\|f(z_{i, j,s})\right\|_{p_s}=\left\|\tilde{h} \circ z_{i, j,s}-z_{i, j,s}\right\|_{p_s} \leq 3 \varepsilon$ holds for $j \in[k]$ and therefore $d_{\mathcal{LP}}\left(\beta_{i}^{\prime}, \beta_{i}\right) \leq$ $(r k)^{3 / 4}\left(3 C^{\prime} \varepsilon\right)^{1 / 2}$ by Lemma \ref{coupdist2}.

We choose now $\left\{n_{i}^{\prime}\right\}_{i=1}^{\infty}$ to be a subsequence of $\left\{n_{i}\right\}_{i=1}^{\infty}$ such that $\beta^{\prime}:=\lim _{i \rightarrow \infty} \beta_{n_{i}^{\prime}}^{\prime}$ exists. Noticing that $\beta^{\prime} \in X_{k}$ and $d_{\mathcal{LP}}\left(\beta^{\prime}, \beta\right) \leq(r k)^{3 / 4}\left(3 C^{\prime} \varepsilon\right)^{1 / 2}$, we get that
$$
d_{\mathcal{LP}}\left(X_{k}, \alpha\right) \leq d_{\mathcal{LP}}\left(\beta^{\prime}, \alpha\right) \leq d_{\mathcal{LP}}\left(\beta^{\prime}, \beta\right)+d_{\mathcal{LP}}(\beta, \alpha) \leq 3(r k)^{3 / 4}\left(C^{\prime} \varepsilon\right)^{1 / 2} .
$$
This inequality holds for arbitrary $\varepsilon>0$ and, hence, we finally obtain $\alpha \in X_{k}$.

\begin{remark}
    This proof works generally for any sequence of multi-$P-$operators with a uniform  bound on their order. However, this proof cannot work for sequences of multi-$P-$operators in which the order of the multi-$P-$operators is diverging.  
\end{remark}

\section{Properties of limit objects}\label{SecPropLimObje}

In this section, we discuss some properties of multi-$P-$operators that are preserved under action convergence. 

\begin{definition}\label{defprops} Let $A\in\mathcal{B}_{r}(\Omega)$ be a multi-$P$-operator. 
\begin{itemize}
\item $A$ is \emph{symmetric} if $$\E[A[v_1,\ldots,v_{r-1}]v_r]=\E[A[v_{\pi(1)}\ldots,v_{\pi(r-1)}]v_{\pi(r)}]$$ holds for every $v_1,\ldots, v_{r}\in L^\infty(\Omega)$ and for every $\pi$ permutation of $[r]$.
\item $A$ is \emph{positivity-preserving} if for every $v_1,\ldots,v_{r-1}\in L^\infty(\Omega)$ with $v_1(x),\ldots,$
$v_{r-1}(x)\geq 0$ for almost every $x\in\Omega$, we have that $(A[v_1,\ldots,v_{r-1}])(x)\geq 0$ holds for almost every $x\in\Omega$.
\item $A$ is \emph{$c$-regular} if $ A[1_\Omega,\ldots,1_\Omega]=c1_\Omega$ for some $c\in\mathbb{R}$.
\item $A$ is a \emph{hypergraphop} if it is positivity-preserving and symmetric.
\item $A$ is \emph{atomless} if $\Omega$ is atomless.
\end{itemize}
\end{definition}

In particular, we notice that the $s-$action of the adjacency tensor of a hypergraph is positivity-preserving and symmetric, i.e.\ a hypergraphop.

\begin{remark}
The $c-$regularity property of a multi-$P-$operator is related to certain regularity properties (i.e.\ having constant degree) of hypergraphs. In particular, we can consider different notions of degrees for hypergraphs.
For a $r-$uniform hypergraph $H=(V,E)$ we define for $s\in [r-1]$ the $s-$degree as
$$
\deg_s(v_1,\ldots,v_s)=\{e\in E: \ v_1,\ldots, v_s\in e\},
$$
for $v_1,\ldots,v_s\in V$ pairwise distinct (compare this degree notion with \eqref{DegHyp} in the following). 
We observe that the $s-$action of the adjacency tensor of an $r-$uniform hypergraph is $c-$regular if and only if the hypergraph has constant $s-$degrees equal to $c$.  
\end{remark}

The following lemmas are generalizations to the multi-linear case of the results from Section 3 in \cite{backhausz2018action} for action convergence and the proofs are similar. 

\begin{lemma}Atomless multi-$P$-operators are closed with respect to $d_M$.
\end{lemma}

\begin{proof}
Let's assume $A\in\mathcal{B}_{r}(\Omega)$ and $B\in\mathcal{B}_{r}(\Omega_2)$ to be two multi-$P$-operators with $d_M(B,A)=d$. Additionally, let's suppose $A\in\mathcal{B}_{r}(\Omega)$ to be atomless. Therefore, there exists a random variable $v\in L^{\infty}_{[-1,1]}(\Omega)$ such that its distribution is uniform on $[-1,1]$. Let's define $\alpha:=\mathcal{D}_A(v,\ldots,v)$.  As $d_H(\cS_1(A),\cS_1(B))\leq 2d$ we have that $\beta=\mathcal{D}_B(w,w^{(2)},\ldots,w^{(r-1)})\in \cS_1(B)$ with $d_{\mathcal{LP}}(\beta,\alpha)\leq 3d$ and thus $d_{\mathcal{LP}}(\alpha_1,\beta_1)\leq3d$, where $\alpha_1=\mathcal{L}(v)=\text{Unif}_{[-1,1]}$ and $\beta_1=\mathcal{L}(w)$ are the marginals of $\alpha$ and $\beta$ on the first coordinate. Thus, the distance $d_{\mathcal{LP}}$ between $\beta_1$  and the uniform distribution is at most $3d$. Therefore, the largest atom in $\beta_1$ is at most $10d$ as by the definition of Levy-Prokhorov distance
$$
\inf\{\delta:\ \beta_1(\{x_0\})\leq \alpha_1(B_{\delta}(x_0))+\delta\}\leq d_{\mathcal{LP}}(\alpha_1,\beta_1)\leq 3d
$$
and $\alpha_1(B_{\delta}(x_0))=2\delta$. Hence the largest atom in $\Omega_2$ has weight at most $10d=10d_M(B,A)$. For this reason, if $B$ is the limit of atomless operators, then $B$ is atomless. 
\end{proof}

Under uniform boundedness conditions, positivity and symmetry of multi-$P-$operators are preserved under action convergence.

\begin{lemma} \label{LemmaClosedSymmetric}Let $p\in[1,\infty]$ and $q\in (1,\infty)$. Let $\{A_i\in\mathcal{B}_{r}(\Omega_i)\}_{i=1}^\infty$ be a sequence of  multi-$P$-operators with a uniform bound on the  $(p,\ldots,p,q)$-norms converging to a multi-$P$-operator $A\in\mathcal{B}_{r}(\Omega)$. If $A_i$ is symmetric for every $i$, then $A$ is also symmetric.
\end{lemma}

\begin{proof} 

Let $\pi$ be a permutation of $[r].$ To show the statement let $v_1,\ldots, v_{r}\in L^\infty_{[-1,1]}(\Omega)$ and let $\mu:=\mathcal{D}_A(v_1,\ldots,v_r)$. By the definition of action convergence, it follows that for every $i\in\mathbb{N}$ there exist functions $v_{i,1},\ldots, v_{i,r}\in L_{[-1,1]}^\infty(\Omega_i)$ such that $\mu_i:=\mathcal{D}_{A_i}(v_{i,1},\ldots,v_{i,r},v_{i,\pi(1)},\ldots,v_{i,\pi(r)})$ weakly converges to $\mu$.  By Lemma \ref{closedlem2}, we have that $\mathbb{E}[v_{i,r}(A_i[v_{i,1},\ldots,v_{i,r-1}])]$ goes to $\mathbb{E}[v_r(A[v_1,\ldots,v_{r-1}])]$ and $\mathbb{E}[v_{i,\pi(r)}(A_i[v_{i,\pi(1)},\ldots,v_{i,\pi(r-1)}])]$ goes to $\mathbb{E}[v_{\pi(r)}(A[v_{\pi(1)},\ldots,$
$v_{\pi(r-1)}])]$ as $i$ goes to infinity. But additionally, we notice that 
$$\mathbb{E}[v_{i,r}(A_i[v_{i,1},\ldots,v_{i,r-1}])]=\mathbb{E}[v_{i,\pi(r)}(A_i[v_{i,\pi(1)},\ldots,v_{i,\pi(r-1)}])]$$
and therefore $$
\mathbb{E}[v_r(A[v_1,\ldots,v_{r-1}])]=\mathbb{E}[v_{\pi(r)}(A[v_{\pi(1)},\ldots,v_{\pi(r-1)}])]
$$
This concludes the proof.
\end{proof} 
\begin{remark}
The $s-$action of the adjacency tensor of a hypergraph is positive and symmetric. 
\end{remark}
Moreover, positivity-preserving and $c-$regular multi-$P-$operators are also closed under action convergence, under slightly different uniform boundedness conditions.

\begin{lemma}\label{LemmaClosedPosPres} Let $p\in [1,\infty),q\in [1,\infty],c\in\mathbb{R}$ and let $\{A_i\in\mathcal{B}_{r}(\Omega_i)\}_{i=1}^\infty$ be a sequence of  multi-$P$-operators with a uniform bound on the  $(p,\ldots,p,q)$-norms converging to a $P$-operator $A\in\mathcal{B}_{r}(\Omega)$. Then we have the following two statements.
\begin{enumerate}
\item If $A_i$ is positivity-preserving for every $i$, then $A$ is also positivity-preserving.
\item If $A_i$ is $c$-regular for every $i$, then $A$ is also $c$-regular.
\end{enumerate}
\end{lemma}
\begin{proof}
To show the first statement, let $v_1,\ldots,v_{r-1}\in L^\infty_{[0,1]}(\Omega)$. By the definition of action convergence, there is a sequence $\{v_{i,1},\ldots,$\ $ v_{i,r-1}\in L_{[-1,1]}^\infty(\Omega_i)\}_{i=1}^\infty$  such that $\mathcal{D}_{A_i}(v_{i,1},\ldots,v_{i,r-1})$ weakly converges to $\mathcal{D}_A(v_1,\ldots,v_{r-1})$ as $i$ goes to infinity. As $\mathcal{L}(v_{i,1},\ldots,v_{i,r-1})$ weakly converges to the non-negative distribution $\mathcal{L}(v_1,\ldots,v_{r-1})$ it follows that $\mathcal{L}(v_{i,1}-|v_{i,{1}}|,\ldots,v_{i,r-1}-|v_{i,{r-1}}|)$ weakly converges to $\delta_0$. Thus, by Lemma \ref{applem2}, we have that 
$$d_{\mathcal{LP}}(\mathcal{D}_{A_i}(v_{i,1},\ldots,v_{i,r-1}),\mathcal{D}_{A_i}(|v_{i,1}|,\ldots,|v_{i,r-1}|))\rightarrow 0$$
for $i\rightarrow \infty$ and, for  this reason, $\mathcal{D}_{A_i}(|v_{i,1}|,\ldots, |v_{i,r-1}|)$ weakly converges to the probability measure $\mathcal{L}(v_1,\ldots,$
$v_{r-1},A[v_1,\ldots,v_{r-1}])$. The fact that $A_i[|v_{i,1}|,\ldots, |v_{i,r-1}|]$ is non-negative for every $i$  directly implies that $A[v_1,\ldots,v_{r-1}]$ is non-negative. 

To show the second statement, let $v_{i,1},\ldots, v_{i,r-1}\in L_{[-1,1]}^\infty(\Omega_i)$ be a sequence of functions such that $\mathcal{D}_{A_i}(v_{i,1},\ldots, v_{i,r-1})$ weakly converges to $\mathcal{D}_A(1_\Omega,\ldots, 1_\Omega)$. We notice that $\mathcal{D}(v_{i,1}-1_{\Omega_i},\ldots,v_{i,r-1}-1_{\Omega_i})$ weakly converges to $\delta_0$ and, for this reason, by Lemma \ref{applem2} we have that 
$$d_{\mathcal{LP}}(\mathcal{D}_{A_i}(1_{\Omega_i},\ldots,1_{\Omega_i}),\mathcal{D}_{A_i}(v_{i,1},\ldots,v_{i,r-1}))\rightarrow 0$$ as $i\rightarrow \infty$. Hence, it follows that $\mathcal{D}_A(1_\Omega,\ldots,1_\Omega)$ is the weak limit of $\mathcal{D}_{A_i}(1_{\Omega_i},\ldots,1_{\Omega_i})$. The result directly follows now by the fact that $ A_i[1_{\Omega_i},\ldots,1_{\Omega_i}]=c1_{\Omega_i}$.
\end{proof}

\begin{remark}
The $s-$action of the adjacency tensor of a hypergraph is positivity-preserving. 
\end{remark}

\section{Norms and metrics comparison}\label{SecNormsMulti}

In this section, we compare different norms and metrics for multi-$P-$operators.\newline

The following two lemmas are generalizations of Lemmas 2.12 and 2.13 in \cite{backhausz2018action}.

\begin{lemma}\label{limlem3} Let $r$ and $k$ be positive integers and let $A,B$ be $r-$th order multi-$P$-operators both in $\mathcal{B}_r(\Omega)$ for some probability space $(\Omega,\mathcal{A},\mu).$ Then $$d_H(\cS_k(A),\cS_k(B))\leq \|A-B\|_{\infty\to 1}^{1/2}(2k)^{3/4}.$$
\end{lemma}
\begin{proof} Let $\mu\in \cS_k(A)$ be arbitrary. We have that there exist functions $v_1,v_2,\dots,v_k\in L^\infty_{[-1,1]}(\Omega)$  such that $\mu$ is equal to the probability measure $\mathcal{D}_A(\{v^{(j)}_i\}_{j\in [r],i\in [k]})$. Let $\nu=\mathcal{D}_B(\{v^{(j)}_i\}_{j\in [r],i\in [k]})\in \cS_k(B)$. Since  $$\|A[v^{(1)}_i,\ldots, v^{(r-1)}_i]-B[v^{(1)}_i,\ldots, v^{(r-1)}_i]\|_1\leq \|A-B\|_{\infty\to 1}\prod_{j\in [r-1]}\|v^{(j)}_i\|_\infty\leq\|A-B\|_{\infty\to 1}$$
holds for every $i\in [k]$, we have by Lemma \ref{coupdist2} that $d_{\mathcal{LP}}(\mu,\nu)\leq \|A-B\|_{\infty\to 1}^{1/2}(2k)^{3/4}$. We obtained that $$\sup_{\mu\in \cS_k(A)}\inf_{\nu\in \cS_k(B)}d_{\mathcal{LP}}(\mu,\nu)\leq \|A-B\|_{\infty\to 1}^{1/2}(2k)^{3/4}.$$
By switching the roles of $A$ and $B$ and repeating the same argument we get the above inequality with $A$ and $B$ switched. This implies the statement of the lemma.
\end{proof}
The following lemma is a direct consequence of Lemma \ref{limlem3}.

\begin{lemma}\label{limspdm}  %{(Norm distance vs. $d_M$ distance)}
Assume that $A,B$ are $r-$th order multi-$P$-operators acting on the same space $L^\infty(\Omega)$. We have  $d_M(A,B)\leq 3\|A-B\|_{\infty\to 1}^{1/2}$.
\end{lemma}
\begin{proof} Using Lemma \ref{limlem3} we obtain that 
\begin{equation*}d_M(A,B)\leq \|A-B\|_{\infty\to 1}^{1/2}\sum_{k=1}^\infty 2^{-k}(2k)^{3/4}\leq 3\|A-B\|_{\infty\to 1}^{1/2}.\end{equation*}
\end{proof}
For a multi-$P-$operator $A\in \cB(\Omega)$ we define the \emph{multi cut norm} as
$$\begin{aligned}
 \|A\|_{\square,\text{multi}}=&\sup _{f^{(1)},\ldots,f^{(r)}\in L_{[0,1]}^{\infty}(\Omega)}\left|\left\langle f^{(r)}, A[f^{(1)},\ldots f^{(r-1)}]\right\rangle\right|.
\end{aligned}
$$

We obtain that for an $r-$th order multi-$P-$operator the $\infty \rightarrow 1$ norm is equivalent to the multi cut norm. This is a generalization of Lemma 8.11 in \cite{LovaszGraphLimits} for graphons.

\begin{lemma}\label{LemmCutInf1} Let $A$ be a multi-$P-$operator $A\in \cB(\Omega).$ The following inequality holds:

$$
\|A\|_{\square,\text{multi}}\leq \|A\|_{\infty\rightarrow 1}\leq 2^r\|A\|_{\square,\text{multi}}.
$$

\end{lemma}
\proof
We get the first inequality by definition:
$$\begin{aligned}
\left\|A\right\|_{\infty \rightarrow 1}&=\sup _{f^{(1)},\ldots,f^{(r-1)}\in L_{[-1,1]}^{\infty}(\Omega)}\left\|A[f^{(1)},\ldots f^{(r-1)}]\right\|_1\\
&=\sup _{f^{(1)},\ldots,f^{(r)}\in L_{[-1,1]}^{\infty}(\Omega)}\left\langle f^{(r)}, A[f^{(1)},\ldots f^{(r-1)}]\right\rangle\\
&=\sup _{f^{(1)},\ldots,f^{(r)}\in L_{[-1,1]}^{\infty}(\Omega)}\left|\left\langle f^{(r)}, A[f^{(1)},\ldots f^{(r-1)}]\right\rangle\right| \\
& \geq\|A\|_{\square,\text{multi}}.
\end{aligned}
$$

We now show the second inequality. We observe the following equality:

$$
\left\|A\right\|_{\infty \rightarrow 1}=\sup _{f^{(1)},\ldots,f^{(r)},g^{(1)},\ldots,g^{(r)}\in L_{[0,1]}^{\infty}(\Omega)}\left\langle f^{(r)}-g^{(r)}, A[f^{(1)}-g^{(1)},\ldots,f^{(r-1)}-g^{(r-1)} ]\right\rangle.
$$

Moreover, for any $f^{(1)},\ldots,f^{(r)},g^{(1)},\ldots,g^{(r)}\in L_{[0,1]}^{\infty}(\Omega)$  we have the following inequality
$$
\begin{aligned}
&\left\langle f^{(r)}-g^{(r)}, A[f^{(1)}-g^{(1)},\ldots,f^{(r-1)}-g^{(r-1)} ]\right\rangle
\\ 
&=\left\langle f^{(r)}, A [f^{(1)}-g^{(1)},\ldots,f^{(r-1)}-g^{(r-1)} ]\right\rangle-\left\langle g^{(r)}, A [f^{(1)}-g^{(1)},\ldots,f^{(r-1)}-g^{(r-1)} ]\right\rangle  \\
&=\left\langle f^{(r)}, A [f^{(1)},f^{(2)}-g^{(2)},\ldots,f^{(r-1)}-g^{(r-1)} ]\right\rangle \\
&\hspace{0.5 cm}-\left\langle f^{(r)}, A [g^{(1)},f^{(2)}-g^{(2)},\ldots,f^{(r-1)}-g^{(r-1)} ]\right\rangle\\
&\hspace{0.5 cm}-\left\langle g^{(r)}, A [f^{(1)},f^{(2)}-g^{(2)},\ldots,f^{(r-1)}-g^{(r-1)} ]\right\rangle \\
&\hspace{0.5 cm}+\left\langle g^{(r)}, A [g^{(1)},f^{(2)}-g^{(2)},\ldots,f^{(r-1)}-g^{(r-1)} ]\right\rangle 
 \\
& \leq 2^r\|A\|_ {\square,\text{multi}}.
\end{aligned}
$$
Therefore, we obtain $\left\|A\right\|_{\infty \rightarrow 1}\leq 2^r\|A\|_ {\square,\text{multi}}. $
\endproof

Let $\varphi:\Omega\to\Omega$ be a bijective measure-preserving transformation. We denote with $\varphi^{-1}$ its measure-preserving inverse. The transformation $\varphi$ induces a natural, linear action on $L^1(\Omega)$, which we also indicate by $\varphi$, defined by $(f)^\varphi(x)=f(\varphi(x))$. Furthermore, for $A\in\mathcal{B}(\Omega)$ let $A^\varphi$ defined as 

$$A^{\varphi}[f^{(1)},\ldots,f^{(r-1)}]=(A[(f^{(1)})^{\varphi},\ldots,(f^{(r-1)})^{\varphi}])^{\varphi^{-1}}.$$

We observe that if $A\in\mathcal{B}(\Omega)$, then $A^\varphi\in\mathcal{B}(\Omega)$ and $d_M(A,A^\varphi)=0$. Let $A$ and $B$ be two $r-$th order multi-$P-$operators such that $A,B\in \cB(\Omega).$ The \emph{multi cut distance} between $A$ and $B$ is defined as
$$\delta_{\square, \text{multi}}(A,B)=\inf_{\varphi,\psi}~\|A^\varphi-B^\psi\|_{\square},$$
where the infimum is taken over all $\varphi,\psi$ invertible measure-preserving transformations from $\Omega$ to $\Omega.$ 

\begin{lemma}\label{limspdm2} Assume that $A,B$ are $r-$th order multi-$P$-operators acting on the same space $\Omega$. Then $d_M(A,B)\leq {\color{black}3\cdot 2^r}\delta_{\square,\text{multi}}(A,B)^{1/2}$.
\end{lemma}

\proof 
The proof follows directly from Lemma \ref{limspdm} and observing that $d_M(A,B)=d_M(A^{\phi},B^{\psi})$ for any bijective and measure preserving transformations $\phi,\psi$ from $\Omega$ to itself.
\endproof

\section{Multi-action convergence of hypergraphs and tensors}\label{SecMultActTensHyp}

We have seen in the previous sections that hypergraphs can be naturally associated with symmetric tensors through, for example, the adjacency tensor. We can therefore study the convergence of sequences of tensors and see the convergence of hypergraphs as a particular case. Moreover, in the previous sections, we noticed that tensors can be associated with multi-linear operators in many different ways. We will compare the notions of convergence induced by the different operators associated to the same tensor. We mainly focus on symmetric tensors as we are originally motivated by undirected hypergraphs. We notice that the obtained notions of convergence are not equivalent. For simplicity, we will mainly present the convergence in the case of $3-$rd order symmetric tensors and therefore hypergraphs with maximal edge cardinality 3. However, the notions of convergence are general and cover tensors of any finite order and hypergraphs with any finite maximal edge cardinality. These convergence notions particularly make sense for uniform hypergraphs. However, we will explain, in Section \ref{NonUnifHypSect}, how one can extend these notions to not lose information regarding the non-maximal cardinality edges in non-uniform hypergraphs.

We recall the notion of $s-$action of a tensor from Section \ref{SecTensHyp}. For a $3$rd-order tensor $T=(T_{i,j,k})_{i,j,k\in [n]}$ the 1-action and the $2-$action of the tensor are respectively 

\begin{equation*}
(T_1[f,g])_{i}=\sum^n_{j,k=1}T_{j,i,k}f_{i}g_{k} 
\end{equation*}and
$$
\begin{aligned}
& (T_2[f,g])_{i,k}=\frac{1}{2}(\sum^n_{j=1}T_{j,i,k}f_{j,i}g_{j,k} +\sum^n_{j=1}T_{j,i,k}f_{j,k}g_{j,i})
\end{aligned}
$$
Therefore, we can interpret the $1-$action as an operator 
$$
\begin{aligned}
T_1:(L^{\infty}([n]))^2\longrightarrow L^{1}([n])
\end{aligned}
$$
and the $2-$action as
$$T_2:(L^{\infty}([n]\times [n],Sym))^2\longrightarrow L^{1}([n]\times [n],Sym)$$
where $Sym$ is the symmetric $\sigma-$algebra on $[n]\times [n]$.

\begin{remark}
More generally the $s-$action of an $r-$th order symmetric tensor $T$ is acting on $r-1$ symmetric $s-$th order tensors and gives as an output another symmetric $s-$th order tensor. For this reason, this $s-$action can be interpreted as an operator
$$
T_s:(L^{\infty}([n]^s,Sym))^{r-1}\longrightarrow L^{1}([n]^s,Sym)
$$
where  $Sym$ is the symmetric $\sigma-$algebra on $[n]^s$.
\end{remark}

In such a way, we obtain two notions of convergence for sequences of $3$rd-order tensors $T_n=(T_{i,j,k})_{i,j,k\in [n]}$, the one obtained by the action convergence of the sequence of multi-linear operators $(T_1)_n$ and the one obtained by the action convergence of the sequence of multi-linear operators $(T_2)_n$.

\begin{remark}
As already pointed out we can associate to an $r-$th order symmetric tensor its $s-$action for $s\in[r-1]$. These different actions can be interpreted as $r-1$ different multi-$P-$operators. Therefore, for a sequence of $r-$th order symmetric tensors, we obtain $r-1$ different notions of convergence. 
\end{remark}

{\color{black}
We will use the results in this section later in this work to connect the metric $d_M$ with other norms and metrics for hypergraph limits.

}

\subsection{Uniform bounds on sequences of $s-$actions}
We recall that in the case of graphs, we typically consider the $1-$action of normalized adjacency matrices. In particular, for dense graphs, we consider

$$
\widetilde{A}(G):L^{\infty}([n])\longrightarrow L^{1}([n])
$$
$$
(\widetilde{A}(G)[f])_i=\sum_{j}\frac{A(G)_{i,j}}{n}f_j.
$$
for a graph $G$ on the vertex set $[n].$ %(compare with Example \ref{ExampMatrixPoper}).  
These linear bounded operators can be easily extended to linear bounded operators 
$$
\widetilde{A}(G):L^1{([n])}\longrightarrow L^{\infty}([n])
$$
and we have that these operators are uniformly bounded (independently by the cardinality of the vertex set $n$) as 
$$
\begin{aligned}
&\|\widetilde{A}(G)[f]\|_{\infty}=\max_i|\sum_{j}\frac{A(G)_{i,j}}{n}f_j|\\&
\leq \max_i\sum_{j}\frac{A(G)_{i,j}}{n}|f_j|\leq\sum_{j}\frac{|f_j|}{n}= \|f\|_1 .\\ &
\end{aligned}
$$
We can observe from the following example that for hypergraphs with maximal edge cardinality $r>2$ this is not true. 

\begin{example}
For example, consider a (dense) hypergraph $H$, its adjacency tensor $(A_{i,j,k})_{i,j,k\in [n]}$ and the associated multi-$P-$operator 
$$(\widetilde{A}(H)[f,g])_{i,k}=\widetilde{A}[f,g]_{i,k}=\frac{1}{2}(\sum^n_{j=1}\frac{A_{j,i,k}}{n}f_{j,i},g_{j,k} +\sum^n_{j=1}\frac{A_{j,i,k}}{n}f_{j,k},g_{j,i})
$$
 and consider the matrices $f,g$ such $$f_{i,j}=g_{i,j}=\begin{cases}
f_{i,j}=0 & \text{ if } i,j\neq 1 \\
f_{i,j}=1 & \text{ if } i=1 \text{ or }j=1
\end{cases}.
$$
\end{example}

However, we can consider a smaller extension.

\begin{lemma}\label{LemmaUnifBoundNormHyp}
    The sequence of operators $$
\widetilde{A}(G_n):L^2{([n]\times [n],Sym)}\times L^2{([n]\times [n],Sym)}\longrightarrow L^{2}([n]\times [n], Sym)
$$
$$(\widetilde{A}(G_n)[f,g])_{i,k}=\widetilde{A}[f,g]_{i,k}=\frac{1}{2}(\sum^n_{j=1}\frac{A_{j,i,k}}{n}f_{j,i},g_{j,k} +\sum^n_{j=1}\frac{A_{j,i,k}}{n}f_{j,k},g_{j,i})
$$
is uniformly bounded in $L^2-$operator norm.
\end{lemma}
\proof
In these spaces, we have a uniform bound, in fact
$$
\begin{aligned}
&
|\widetilde{A}[f,g]_{i,k}|\leq \frac{1}{2}(\sum^n_{j=1}\frac{A_{j,i,k}}{n}|f_{j,i}g_{j,k}| +\sum^n_{j=1}\frac{A_{j,i,k}}{n}|f_{j,k}g_{j,i}|)\\&
\leq \frac{1}{2}(\sum^n_{j=1}\frac{1}{n}|f_{j,i}g_{j,k}| +\sum^n_{j=1}\frac{1}{n}|f_{j,k}g_{j,i}|) 
\\ & \leq
\frac{1}{2}((\sum^n_{j=1}\frac{1}{n}|f_{j,i}|^2)^{\frac{1}{2}}(\sum^n_{j=1}\frac{1}{n}|g_{j,k}|^2 )^{\frac{1}{2}}+(\sum^n_{j=1}\frac{1}{n}|f_{j,k}|^2)^{\frac{1}{2}}(\sum^n_{j=1}\frac{1}{n}|g_{j,i}|^2)^{\frac{1}{2}})
\end{aligned}
$$

where the last inequality follows by Cauchy-Schwartz inequality. Therefore, we obtain

{\color{black}

$$
\begin{aligned}
&
\|\widetilde{A}[f,g]\|_2\leq
(\frac{1}{n^2}\sum_{i,k=1}
|\widetilde{A}[f,g]_{i,k}|^2)^{\frac{1}{2}}
\\&\leq
\frac{1}{2}\left(\frac{1}{n^2}\sum_{i,k=1}\left((\sum^n_{j=1}\frac{1}{n}|f_{j,i}|^2)^{\frac{1}{2}}(\sum^n_{j=1}\frac{1}{n}|g_{j,k}|^2 )^{\frac{1}{2}}+(\sum^n_{j=1}\frac{1}{n}|f_{j,k}|^2)^{\frac{1}{2}}(\sum^n_{j=1}\frac{1}{n}|g_{j,i}|^2)^{\frac{1}{2}}\right)^2\right)^{\frac{1}{2}}
\\&\leq
\frac{1}{2}\left(\frac{1}{n^2}\sum_{i,k=1}(\sum^n_{j=1}\frac{1}{n}|f_{j,i}|^2)(\frac{1}{n}\sum^n_{j=1}|g_{j,k}|^2 )\right)^{\frac{1}{2}}+\left(\frac{1}{n^2}\sum_{i,k=1}(\sum^n_{j=1}\frac{1}{n}|f_{j,k}|^2)(\sum^n_{j=1}\frac{1}{n}|g_{j,i}|^2)\right)^{\frac{1}{2}}
\\&\leq
\frac{1}{2}\left((\frac{1}{n^2}\sum^n_{j,i=1}|f_{j,i}|^2)^{\frac{1}{2}}(\frac{1}{n^2}\sum^n_{j,k=1}|g_{j,k}|^2 )^{\frac{1}{2}}+(\frac{1}{n^2}\sum^n_{j,k=1}|f_{j,k}|^2)^{\frac{1}{2}}(\frac{1}{n^2}\sum^n_{j,i=1}|g_{j,i}|^2)^{\frac{1}{2}}\right)\\
&  =\|f\|_{2}\|g\|_{2}
\end{aligned}
$$
where in the third inequality we used Minkowski inequality.

\endproof
\begin{remark}
More in general, for $r>2$, for a sequence of dense hypergraphs the sequence of $(r-1)-$actions of the relative normalized adjacency tensors cannot be extended/interpreted as a uniformly bounded sequence of linear operators from $L^1\times \ldots \times L^1$ to $L^{\infty}$. Therefore, one has to consider them as operators from $L^{p_1}\times \ldots \times L^{p_{r-1}}$ to $L^q$ with $p_1,\ldots,p_{r-1}\neq 1$ and $q\neq \infty$. This happens already in the case of graph limits for sparse graph sequences, and we know that this translates in larger classes of measures admitting also more irregular measures, possibly not absolutely continuous with respect to the Lebesgue measure on the interval $[0,1]$. %This could be also the reason for the additional variables needed to represent hypergraphons. 
Instead, for every $r$ the sequence of $1-$actions of the normalized adjacency matrices of dense graphs is a uniformly bounded sequence of linear operators from $L^1\times \ldots \times L^1$ to $L^{\infty}$.
\end{remark}

\begin{remark}
The same estimates hold for the non-symmetrized $s-$action.
\end{remark}

\subsection{Properties of $s-$actions as $P-$operators}

We underline here a few more properties of the action of (normalized) adjacency matrices of hypergraphs and, therefore, also of their limits by Lemma \ref{LemmaClosedSymmetric} and Lemma \ref{LemmaClosedPosPres}. 

First of all, we notice that the actions of (normalized) adjacency tensors are obviously positivity-preserving multi-$P-$operators and, therefore, their action convergence limits are too.
The following lemma and remark state that the action of a symmetric tensor is a symmetric multi-$P-$operator. 
\begin{lemma}
For a $3-$rd order symmetric tensor $T=((T)_{i,j,k})_{i,j,k\in [n]}$ the multi-$P-$operator $(T)_2$ is symmetric. 
\end{lemma}
\proof
The result follows from the following equality
\begin{equation}
\begin{aligned}
\mathbb{E}[(T)_2[f,g]h]&=\frac{1}{n^2}\sum^n_{i,k=1}\frac{1}{2}(\sum^n_{j=1}T_{i,j,k}f_{i,j}g_{j,k}+\sum^n_{j=1}T_{i,j,k}g_{i,j}f_{j,k})h_{i,k}\\
&=\frac{1}{n^2}\sum^n_{i,j,k=1}T_{i,j,k}f_{i,j}g_{j,k}h_{i,k}\\
&=\mathbb{E}[(T)_2[h,g]f]\\
&=\mathbb{E}[(T)_2[f,h]g].
\end{aligned}
\end{equation}
\endproof
\begin{remark}
Similarly, the $s-$action of a symmetric $r-$th order $n-$dimensional symmetric tensor $T$ is symmetric for every $s\in [r-1]$ by a similar computation.
\end{remark}

Therefore, the limit of the sequence of symmetric tensors will also be symmetric for Lemma \ref{LemmaClosedSymmetric}.

\subsection{Generalization of $s-$actions}\label{SecSactioGen}

Recall the $1-$actions introduced in \ref{DefAction} for tensors. In this section, we generalize the notion of $1-$action to $r-$kernels (see below) and study its properties. We will also present a natural generalization of $s-$action, for $s\in [r],$ to $r-$kernels.

Let $\Omega$ be a probability space. We call a measurable function 
$$
W:\Omega^{r}\rightarrow \R
$$
such that $\|W\|_1<\infty$ an \emph{$r-$kernel}. 

We will say that an $r-$kernel $W$ is an \emph{$r-$graphon} if $W$ takes values in $[0,1].$

\begin{remark}
    This is a trivial generalization of real-valued graphons \cite{LovaszGraphLimits}. In particular, for $r=2$ we have that the $r-$graphons are the real-valued graphons.
\end{remark}

\begin{remark}
An $r-$th order $n-$dimensional tensor is an $r-$kernel where $\Omega=[n],$ endowed with the uniform measure. One can also naturally represent a tensor with a $r-$kernel that is a step-function (as a trivial generalization of the step-representation of a graph (or matrix) for real-valued graphons).
\end{remark}

We can identify an $r-$kernel $W$ with its $1-$action, the $r-$th order multi-$P-$operator $$
(W)_1:L^{\infty}(\Omega)^{r-1}\rightarrow L^{1}(\Omega)
$$
defined as $$
(W)_1[f^{(1)},\ldots f^{(r-1)}]=\int_{\Omega^{r-1}}W(x_1,\ldots,x_k)f^{(1)}(x_1)\ldots f^{(r-1)}(x_{r-1}) \mathrm{d}x_1\ldots \mathrm{d}x_{r-1}.
$$

For a $k-$kernel $W$ we can define the $1-$cut norm as
$$
\begin{aligned}
    \|W\|_{\square_1}&=\sup_{ f^{(1)},\ldots,f^{(r)}:\Omega\rightarrow[0,1]}\left|\int_{\Omega^r}W(x_1,\ldots,x_r)f^{(1)}(x_1)\ldots f^{(r)}(x_r) \mathrm{d}x_1\ldots \mathrm{d}x_r\right|\\
    &=\sup _{f^{(1)},\ldots,f^{(r)}\in L_{[0,1]}^{\infty}(\Omega)}\left\langle f^{(r)}, {\color{black}(W)_1}[f^{(1)},\ldots f^{(r-1)}]\right\rangle= \|(W)_1\|_{\square,\text{multi}}
    \end{aligned}$$

Compare also \cite{HypergraphonsZhao}.

From Lemma \ref{LemmCutInf1} we directly obtain that for an $r-$kernel $W$ the $\infty \rightarrow 1$ norm of the associated multi-$P-$operator $(W)_1$ is equivalent to the $1-$cut norm. 
\begin{equation}
    \|W\|_{\square_1}\leq \|(W)_1\|_{\infty\rightarrow 1}\leq 2^r\|W\|_{\square_1}.
\end{equation}
This is a generalization of Lemma 8.11 in \cite{LovaszGraphLimits} for graphons.

For a bijective measure-preserving transformation $\varphi:\Omega\rightarrow \Omega$ and an $r-$kernel $W,$ we denote with $W^{\varphi}$ the $r-$kernel defined for every $x_1,\ldots,x_r\in \Omega$ as
$$
W^{\varphi}(x_1,\ldots,x_r)=W(\varphi(x_1),\ldots,\varphi(x_r)).
$$
We observe that $(W)_1^{\varphi}=(W^{\varphi})_1.$
Moreover, for two $r-$kernels $W$ and $U$ on the same probability space $\Omega$, we define the $1-$cut metric
\begin{equation}
    \begin{aligned}
        \delta_{\square_1}(U,W)&=\inf_{\varphi,\psi}\|W^{\varphi}-U^{\psi}\|_{\square_{1}}\\
        & =\inf_{\varphi,\psi}\|(W)_1^{\varphi}-(U)_1^{\psi}\|_{\square,\text{multi}} \\
        & =\delta_{\square, \text{multi}}((U)_1,(W)_1).
    \end{aligned}
\end{equation}

Therefore, from Lemma \ref{limspdm2}, we obtain 
\begin{equation}
    d_M((W)_1,(U)_1)\leq {\color{black} 3 \cdot 2^r}\delta_{\square_1}(W,U)^{1/2}.
\end{equation}
This implies that convergence in the $1-$cut metric (or $1-$cut norm) of a sequence of $r-$kernels implies multi-linear action convergence of the sequence of the $1-$actions associated with the $r-$kernels. \newline

\begin{remark}
Similarly, we can consider the $s-$action of a symmetric $r-$kernel $W$ as the straightforward generalization of the $s-$action of a symmetric tensor to $r-$kernels. For brevity, we write down explicitly only the $2-$action for a symmetric $3-$kernel $W$ that is the multi-$P-$operator
$$
(W)_2:L^{\infty}(\Omega)^{2}\rightarrow L^{1}(\Omega)
$$

$$
(W)_2[f,g]=\frac{1}{2}\left(\int_{\Omega^2}W(x,y,z)f(x,y)g(y,z)\mathrm{d}y+\int_{\Omega^2}W(x,y,z)f(z,y)g(y,x)\mathrm{d}y\right).
$$
Let's now consider the (too) strong $2-$cut norm 
$$\begin{aligned}\|W\|_{\square_{2,\text{TS}}}&=\sup _{\substack{f, g, h:\left[0,1]^2 \rightarrow[0,1]\right. \\  \text { symmetric }}}\left|\int_{[0,1]^3} W(x, y, z) f(x, y) g(x, z) h(y, z) \mathrm{d} x\mathrm{d} y \mathrm{d} z\right|\\
&=\|(W)_2\|_{\square,\text{multi}}.\end{aligned}$$

Therefore, we can use the reasoning used for $(W)_1,$ substituting $(W)_1$ with $(W)_2$ and $\|W\|_{\square_{1}}$ with $\|W\|_{\square_{2,\text{TS}}},$ to obtain 
\begin{equation}
    d_M((W)_2,(U)_2)\leq {\color{black} 3 \cdot 2^r}\delta_{\square_{2,TS}}(W,U)^{1/2}={\color{black} 3 \cdot 2^r}(\inf_{\varphi,\psi}\|W^{\varphi}-U^{\psi}\|_{\square_{2,\text{TS}}})^{1/2}.
\end{equation}
The (too) strong $2-$cut norm has been studied in the context of hypergraph limits before. The interested reader can find more information about it in Section 3 of \cite{HypergraphonsZhao}. There it is also explained that many interesting hypergraph sequences do not admit a convergent subsequence in this norm.

\end{remark}

{\color{black}
Therefore, from the results in this section, we directly get examples of convergent hypergraph limits in $d_M,$ see the next section or Section 3 in \cite{HypergraphonsZhao}.
}
\subsection{Examples of hypergraph sequences and action convergence}

The emergence of multiple operators and therefore of different notions of convergence of symmetric tensors is related to the emergence of different levels of quasi-randomness for sequences of hypergraphs \cite{RandomneLimitTow,QuasirandHyp1,QuasirandHyp2}.

We illustrate here this relationship with some examples.

\begin{example}\label{ExamCompERHyp}
    Let's consider the $3-$uniform Erdős–Rényi hypergraph $G(n,\frac{1}{8},3)$ from Example \ref{ERRandomHypergraph} and the $3-$uniform hypergraph $T(n,\frac{1}{2})$, i.e. the $3-$uniform hypergraph with the triangles of an Erdős–Rényi graph $G(n,\frac{1}{2},2)$ from Example \ref{TriangRandHyp} as edges. We now consider the sequence $(A_n)_{n\in \mathbb{N}}$ of the normalized adjacency tensors associated with (a realization of) $G(n,\frac{1}{8},3)$, i.e.\  $A_n=\sfrac{A(G(n,\frac{1}{8},3))}{n}$, and the sequence $(B_n)_{n\in \N}$ of the normalized adjacency tensors associated with (a realization of) $T(n,\frac{1}{2})$, i.e.\ $B_n=\sfrac{A(T(n,\frac{1}{2}))}{n}$. We remark that the normalization of the adjacency tensors we are choosing is necessary to satisfy the hypothesis of Theorem \ref{CompactnesMultiActConv} and, therefore, to ensure a convergent (sub)sequence as shown in \ref{LemmaUnifBoundNormHyp}. However, different normalizations could be chosen as we will explore later.

If we now consider the sequences of multi-$P-$operators $(A_n/n)_1$ and $(B_n/n)_1$ they both action converge to the same limit object, the $1-$action of the constant $3-$graphon $W=1/8$ defined on $[0,1]\times [0,1]\times [0,1],$ where the unit interval is endowed with the Lebesgue measure. This can be easily seen using the results in Section  \ref{SecSactioGen} and known facts about these random hypergraph models and the $1-$cut norm $\|\cdot\|_{\square_1}$ (see Section 3 in \cite{HypergraphonsZhao}).  However, the two random hypergraph models are very different. To see the combinatorial difference between these two random hypergraph models consider how many edges can be present in an induced hypergraph on $4$ vertices. In $T(n,\frac{1}{2})$ there cannot be exactly three edges but in $G(n,\frac{1}{8},3)$ this happens with probability $$4\cdot \frac{1}{8}\cdot \frac{1}{8}\cdot \frac{1}{8}\cdot \frac{7}{8}=\frac{7}{1024}.$$

Instead, if we now consider the sequences of multi-$P-$operators $(A_n)_2$ and $(B_n)_2$ the two sequences are now converging to two different limits as we show in Lemma \ref{LemmaDiffLimitERHTriangER} below. Again one can easily see using the results in Section  \ref{SecSactioGen} and known facts about these random hypergraph models and the $2-$cut norm $\|\cdot\|_{\square_{2,\text{TS}}}$ (see Section 3 in \cite{HypergraphonsZhao} again) that the sequence of multi-$P-$operators $(A_n)_2$ converges to the $2-$action of the $3-$graphon $W=1/8$ defined on $[0,1]\times [0,1]\times [0,1],$ where the unit interval is endowed with the Lebesgue measure. However, we cannot use the same method to say something about the sequence $(B_n)_2$ as $B_n$ is not convergent in $\|\cdot\|_{\square_{2,\text{TS}}}.$

\end{example}

\begin{lemma}\label{LemmaDiffLimitERHTriangER}
The (sub-)sequences of the multi-$P-$operators $(A_n)_2$ and $(B_n)_2$, as defined in Example \ref{ExamCompERHyp}, have different action convergence limits.
\end{lemma}
\proof
Let's denote with $E_n$ the set of edges of (a realization of) the Erdős–Rényi graph $G(n,\frac{1}{2})$ from which $T(n,\frac{1}{2})$ is generated, that is the Erdős–Rényi graph from which the triangles are taken to create the edges of $T(n,\frac{1}{2})$.
Let $\mathbbm{1}_{\Omega_n}$ be the $n \times n$ matrix with every entry equal to $1$.  We can observe that the distribution

$$\mathcal{L}(\mathbbm{1}_{\Omega_n},\mathbbm{1}_{\Omega_n},(A_n)_2[\mathbbm{1}_{\Omega_n},\mathbbm{1}_{\Omega_n}])$$

of the $3-$random vector 
$$(\mathbbm{1}_{\Omega_n},\mathbbm{1}_{\Omega_n},(A_n)_2[\mathbbm{1}_{\Omega_n},\mathbbm{1}_{\Omega_n}])
$$where 

$$(A_n)_2[\mathbbm{1}_{\Omega_n},\mathbbm{1}_{\Omega_n}]_{i,k}=\sum^n_{j=1}(A_n)_{j,i,k}(\mathbbm{1}_{\Omega_n})_{j,i},(\mathbbm{1}_{\Omega_n})_{j,k} =\sum^n_{j=1}(A_n)_{j,i,k}$$
and the distribution
$$\mathcal{L}(\mathbbm{1}_{\Omega_n},\mathbbm{1}_{\Omega_n},(B_n)_2[\mathbbm{1}_{\Omega_n},\mathbbm{1}_{\Omega_n}])$$of the $3-$random vector 
$$(\mathbbm{1}_{\Omega_n},\mathbbm{1}_{\Omega_n},(B_n)_2[\mathbbm{1}_{\Omega_n},\mathbbm{1}_{\Omega_n}])
$$where 
$$(B_n)_2[\mathbbm{1}_{\Omega_n},\mathbbm{1}_{\Omega_n}]_{i,k}=\sum^n_{j=1}(B_n)_{j,i,k}(\mathbbm{1}_{\Omega_n})_{j,i},(\mathbbm{1}_{\Omega_n})_{j,k} =\sum^n_{j=1}(B_n)_{j,i,k}$$are very different. 
{\color{black}
In fact, for $n\rightarrow \infty$ we have that for any $(i,k)\in [n]\times [n]\setminus \{(i,i):\ i\in[n]\}$ for any $j\in [n]$ the probability that $\{i,j,k\}$ is an edge of $G(n,\frac{1}{8},3)$ is $\frac{1}{8}$. Therefore, $(A_n)_2[\mathbbm{1}_{\Omega_n},\mathbbm{1}_{\Omega_n}]_{i,k}$ is a sum of $n$ Bernoulli (almost) independent random variables with parameter $1/8$ divided by $n.$ Therefore, by (a standard argument using) the law of large numbers we obtain that 

$$\mathcal{L}(\mathbbm{1}_{\Omega_n},\mathbbm{1}_{\Omega_n},(A_n)_2[\mathbbm{1}_{\Omega_n},\mathbbm{1}_{\Omega_n}])\rightarrow\delta_{(1,1,\frac{1}{8})}$$
However, similarly, we obtain that
$$\mathcal{L}(\mathbbm{1}_{\Omega_n},\mathbbm{1}_{\Omega_n},(B_n)_2[\mathbbm{1}_{\Omega_n},\mathbbm{1}_{\Omega_n}])\rightarrow\frac{1}{2}\delta_{(1,1,0)}+\frac{1}{2}\delta_{(1,1,\frac{1}{4})}$$
as if $(i,k)\in E_n$ then for any $j\in [n]$ the probability that $\{i,j,k\}$ is an edge of  $T(n,\frac{1}{2})$ is $\frac{1}{4}$ but if $(i,k)\notin E_n$ then there is no $j\in [n]$ such that $\{i,j,k\}$ is an edge of  $T(n,\frac{1}{2})$.
Therefore, the $1-$profiles $\cS_1(A)$ and $\cS_1(B)$ of the action convergence limits  $A$ and $B$ (passing to subsequences if it is necessary) of the sequences $((A_n)_2)_n$ and $((B_n)_2)_n$ are at Hausdorff distance bigger than a constant $c>0$. Let's suppose by contradiction that there exists $f,g\in L_{[-1,1]}^{\infty}(\Omega)$ such that for every $\varepsilon>0$ $$d_{\mathcal{LP}}(\mathcal{L}(f,g,A[f,g]),\mathcal{L}(\mathbbm{1}_{\Omega},\mathbbm{1}_{\Omega},B[\mathbbm{1}_{\Omega},\mathbbm{1}_{\Omega}]))\leq \varepsilon.$$}
We recall that convergence in distribution to a constant and convergence in probability to the same constant are equivalent and, as the random variables are bounded between $1$ and $-1$, convergence in probability is equivalent to the convergence of the $p-$th moment. Therefore, for any $\delta >0$, we can choose $\varepsilon$ small enough such that 
$$
\|\mathbbm{1}_{\Omega}- f\|_1\leq\|\mathbbm{1}_{\Omega}- f\|_{p_1}< \delta \text{ and } \|\mathbbm{1}_{\Omega}- g\|_1\leq\|\mathbbm{1}_{\Omega}- g\|_{p_2}< \delta 
$$and, therefore, 

$$\|A[\mathbbm{1}_{\Omega},\mathbbm{1}_{\Omega}]-A[f,g]\|_1\leq\|A[\mathbbm{1}_{\Omega},\mathbbm{1}_{\Omega}]-A[f,g]\|_q< 2C\delta
$$
Using Lemma \ref{coupdist}, we obtain that 
 $$d_{\mathcal{LP}}(\mathcal{L}(f,g,A[f,g]),\mathcal{L}(\mathbbm{1}_{\Omega},\mathbbm{1}_{\Omega},A[\mathbbm{1}_{\Omega},\mathbbm{1}_{\Omega}]))\leq 3^{\frac{3}{4}}\delta^{\frac{1}{2}}\max\{(2C)^{\frac{1}{2}},1\}$$
Therefore, for the triangular inequality we have 
$$\begin{aligned}
d_{\mathcal{LP}}(\mathcal{L}(f,g,A[f,g]),&\mathcal{L}(\mathbbm{1}_{\Omega},\mathbbm{1}_{\Omega},B[\mathbbm{1}_{\Omega},\mathbbm{1}_{\Omega}]))
 \\
&\geq \big|d_{\mathcal{LP}}(\mathcal{L}(\mathbbm{1}_{\Omega},\mathbbm{1}_{\Omega},B[\mathbbm{1}_{\Omega},\mathbbm{1}_{\Omega}]),\mathcal{L}(\mathbbm{1}_{\Omega},\mathbbm{1}_{\Omega},A[\mathbbm{1}_{\Omega},\mathbbm{1}_{\Omega}])) \\
&\hspace{0.4cm}-d_{\mathcal{LP}}(\mathcal{L}(f,g,A[f,g]),\mathcal{L}(\mathbbm{1}_{\Omega},\mathbbm{1}_{\Omega},A[\mathbbm{1}_{\Omega},\mathbbm{1}_{\Omega}]))\big|\\
&\geq K - 3^{\frac{3}{4}}\delta^{\frac{1}{2}}\max\{(2C)^{\frac{1}{2}},1\}\geq c >0
\end{aligned}$$
where $K>0$ and $\delta \rightarrow 0$ as $\varepsilon \rightarrow 0 $. But this is in contradiction with 
$$d_{\mathcal{LP}}(\mathcal{L}(f,g,A[f,g]),\mathcal{L}(\mathbbm{1}_{\Omega},\mathbbm{1}_{\Omega},B[\mathbbm{1}_{\Omega},\mathbbm{1}_{\Omega}]))\leq \varepsilon.$$
\endproof

\begin{remark}
We could have deduced directly that the weak limit of
$$\mathcal{L}(\mathbbm{1}_{\Omega_n},\mathbbm{1}_{\Omega_n},(A_n)_2[\mathbbm{1}_{\Omega_n},\mathbbm{1}_{\Omega_n}])\rightarrow\delta_{(1,1,\frac{1}{8})}$$
as we know the limit constant $3-$graphon $W=1/8$ on $[0,1]^3$ of $(A_n)_2.$ Observe in fact that 
$$\mathcal{L}(\mathbbm{1}_{[0,1]},\mathbbm{1}_{[0,1]},(W)_2[\mathbbm{1}_{[0,1]},\mathbbm{1}_{[0,1]})=\delta_{(1,1,\frac{1}{8})}.$$
\end{remark}

\subsection{Finite hypergraphs and action convergence}

Now that we have given some motivating examples for sequences of hypergraphs with diverging number of vertices we study what action convergence and the $k-$profiles capture for finite tensors and hypergraphs.  

The following theorem states that finite tensors are completely determined by the action convergence distance, up to relabelling of the indices. This is particularly interesting for adjacency tensors of hypergraphs because the following result implies that two adjacency tensors of two (finite) hypergraphs are identified if and only if the two hypergraphs are isomorphic.    

\begin{theorem}\label{IdentificationTensors}
For two $3-$rd order $n-$dimensional symmetric tensors $T=(T_{i,j})_{i,j\in [n]}$ and $(\widetilde{T})_{i,j\in [n]}$, the $2-$actions $T_2$ and $\widetilde{T}_2$ are at distance zero in action convergence distance $d_M$ if and only if there exists a bijective map $$\psi:[n]\rightarrow [n]$$
  such that 
   $$T_{i,j,k}=\widetilde{T}_{\psi(i),\psi(j),\psi(k)}.$$  \end{theorem}
\proof
The only non-trivial implication is the “only if” part.
We observe that, in the finite case, it must exist a bijective and measure-preserving function
$$\begin{aligned}
&\phi: ([n]\times [n],Sym)\longrightarrow ([n]\times [n],Sym)
\\ & (i,k)\mapsto\phi(i,k)=(\phi_1(i,k),\phi_2(i,k))
\end{aligned}$$
such that 
$$
(T_2[f,g])^{\phi}=\widetilde{T}_2[f^{\phi},g^{\phi}])
$$
for all symmetric matrices $f,g$ on $[n]\times[n]$. 

Because, in general, to have 

$$
\cL(f,g, T_2[f,g])= \cL(f^{\phi},g^{\phi},  (\widetilde{T}_2[f^{\phi},g^{\phi}]))
$$
we need 
$$(T_2[f,g])^{\phi}=\widetilde{T}_2[f^{\phi},g^{\phi}]).$$

Therefore, we can compare the two terms
$$
\begin{aligned}
(T_2[f,g])_{i,k}^{\phi}=\frac{1}{2}(\sum^n_{j=1}T_{j,\phi_1(i,k),\phi_2(i,k)}f_{j,\phi_1(i,k)}g_{j,\phi_2(i,k)} +\sum^n_{j=1}T_{j,\phi_2(i,k),\phi_1(i,k)}f_{j,\phi_2(i,k)}g_{j,\phi_1(i,k)})
\end{aligned}
$$and
$$
\widetilde{T_2}[f^{\phi},g^{\phi}]_{i,k}=\frac{1}{2}(\sum^n_{j=1}\widetilde{T}_{j,i,k}f_{\phi_1(j,i),\phi_2(j,i)}g_{\phi_1(j,k),\phi_2(j,k)} +\sum^n_{j=1}\widetilde{T}_{j,i,k}f_{\phi_1(j,k),\phi_2(j,k)}g_{\phi_1(j,i),\phi_2(j,i)}).
$$
Now, we choose $f=\mathbbm{1}_{\{\phi_1(i,k),a\}}$ and $g=\mathbbm{1}_{\{a,\phi_2(i,k)\}}$ where $\mathbbm{1}_{\{c,d\}}$ is the indicator function of the set $\{(c,d),(d,c)\}$. Then we have 
$$
\begin{aligned}
&(T_2[f,g])_{i,k}^{\phi}=\frac{1}{2}(\sum^n_{j=1}T_{j,\phi_1(i,k),\phi_2(i,k)}{\mathbbm{1}_{\{\phi_1(i,k),a\}}}_{j,\phi_1(i,k)}{\mathbbm{1}_{\{a,\phi_2(i,k)\}}}_{j,\phi_2(i,k)} \\
&+\sum^n_{j=1}T_{j,\phi_2(i,k),\phi_1(i,k)}{\mathbbm{1}_{\{\phi_1(i,k),a\}}}_{j,\phi_2(i,k)}{\mathbbm{1}_{\{a,\phi_2(i,k)\}}}_{j,\phi_1(i,k)})=\\
&\frac{1}{2}T_{a,\phi_1(i,k),\phi_2(i,k)}
\end{aligned}
$$and
\begin{equation*}
\begin{aligned}
&\widetilde{T_2}[f^{\phi},g^{\phi}]_{i,k}=\frac{1}{2}(\sum^n_{j=1}\widetilde{T}_{j,i,k}{\mathbbm{1}_{\{\phi_1(i,k),a\}}}_{\phi_1(j,i),\phi_2(j,i)}{\mathbbm{1}_{\{a,\phi_2(i,k)\}}}_{\phi_1(j,k),\phi_2(j,k)}  \\ &+ \sum^n_{j=1}\widetilde{T}_{j,i,k}{\mathbbm{1}_{\{\phi_1(i,k),a\}}}_{\phi_1(j,k),\phi_2(j,k)}{\mathbbm{1}_{\{a,\phi_2(i,k)\}}}_{\phi_1(j,i),\phi_2(j,i)}).
\end{aligned}
\end{equation*}
From the second expression, we can notice that for an element of the sum to be non-zero it is necessary that one of the following sets of conditions is satisfied:

\begin{equation}\label{eqn:Cond1nonzero}
\begin{aligned}
&\phi_1(i,k)=\phi_1(d,i)\\
&a=\phi_2(d,i)=\phi_2(d,k)\\
&\phi_2(i,k)=\phi_1(d,k)
\end{aligned}
\end{equation}
\begin{equation}\label{eqn:Cond2nonzero}
\begin{aligned}
&\phi_1(i,k)=\phi_2(d,i)\\
&a=\phi_1(d,i)=\phi_2(d,k)\\
&\phi_2(i,k)=\phi_1(d,k)
\end{aligned}
\end{equation}

\begin{equation}\label{eqn:Cond3nonzero}
\begin{aligned}
&\phi_1(i,k)=\phi_1(d,i)\\
&a=\phi_2(d,i)=\phi_1(d,k)\\
&\phi_2(i,k)=\phi_2(d,k)
\end{aligned}
\end{equation}
\begin{equation}\label{eqn:Cond4nonzero}
\begin{aligned}
&\phi_1(i,k)=\phi_2(d,i)\\
&a=\phi_1(d,i)=\phi_1(d,k)\\
&\phi_2(i,k)=\phi_2(d,k)
\end{aligned}
\end{equation} 
We observe that varying $a$ we accordingly vary $d$ as $\phi$ is bijective. In fact, for all conditions \eqref{eqn:Cond1nonzero}, \eqref{eqn:Cond2nonzero}, \eqref{eqn:Cond3nonzero} and \eqref{eqn:Cond4nonzero} if there would be two distinct $d$ and $\tilde{d}$ in $[n]$ corresponding to the same $a$ then $\phi$ would fail to be bijective. For this reason, we obtain from the conditions \eqref{eqn:Cond1nonzero},\eqref{eqn:Cond2nonzero},\eqref{eqn:Cond3nonzero} and \eqref{eqn:Cond4nonzero} that $\phi_1$ (respectively $\phi_2$) depend only on the second variable. Moreover, we notice that a necessary condition to be bijective and measure-preserving (measurable) for $\phi$ is 
\begin{equation}\label{eqn:condMeasPreBiject}
\begin{aligned}
&\phi_1(i,k)=\phi_2(k,i).\\
%&\phi_2(i,k)=\phi_1(k,i)
\end{aligned}
\end{equation}

Therefore, we notice that conditions \eqref{eqn:Cond2nonzero} and \eqref{eqn:Cond3nonzero} would contradict condition \eqref{eqn:condMeasPreBiject}. In conclusion, we can only have from \eqref{eqn:Cond1nonzero} and \eqref{eqn:condMeasPreBiject} that 

$$
\phi_1(i,j)=\psi(j)
$$
$$
\phi_2(i,j)=\psi(i)
$$
or from \eqref{eqn:Cond4nonzero} and \eqref{eqn:condMeasPreBiject} that 
$$
\phi_1(i,j)=\psi(i)$$
$$
\phi_2(i,j)=\psi(j).
$$
where $\psi$ is a permutation of $[n].$
Therefore, substituting and requiring that $$
(T_2[f,g])^{\phi}=\widetilde{T_2}[f^{\phi},g^{\phi}]
$$
we obtain that 

$$T_{\psi(d),\psi(i),\psi(k)}=T_{a,\psi(i),\psi(k)}=T_{a,\phi_1(i,k),\phi_2(i,k)}=2(T_2[f,g])_{i,k}^{\phi}=2\widetilde{T_2}[f^{\phi},g^{\phi}]_{i,k}=\widetilde{T}_{d,i,k}.$$
\endproof
This result holds more generally as explained in the following remark.
\begin{remark}
We can use the same reasoning as in the proof of Theorem \ref{IdentificationTensors} to show more generally that the $r-1$-actions of two $r-$th order symmetric tensors $T=(T_{i_1,\ldots,i_r})_{i_1,\ldots,i_r\in [n]}$ and $\widetilde{T}=(\widetilde{T}_{i_1,\ldots,i_r})_{i_1,\ldots,i_r\in [n]}$ are completely determined by the action convergence distance, i.e.\ their $(r-1)-$actions are at action convergence distance $d_M$ zero if and only if  

$$
T_{\psi(i_1),\ldots,\psi(i_{r})}=\widetilde{T}_{i_1,\ldots,i_r}.
$$
In fact, similarly to the case $r=3$, there must exist a  bijective and measure-preserving transformation 
$$\phi=(\phi_1,\ldots,\phi_{r}):([n]^{r-1},Sym)\longrightarrow ([n]^{r-1},Sym)$$
such that $$
((T)_{r-1}[f_1,\ldots,f_{r-1}])^{\phi}=(\widetilde{T})_{r-1}[f^{\phi}_1,\ldots,f^{\phi}_{r-1}]
$$
for all $f_1,\ldots,f_{r-1}$ symmetric $(r-1)-$th order tensors and where for a symmetric $(r-1)-$th order tensor $f$ we define $$f^{\phi}(i_1,\ldots,i_{r-1})=f(\phi_1(i_1),\ldots,\phi_{r-1}(i_{r-1})).$$
Moreover, using the test functions $f_s=\mathbbm{1}_{\{a,\phi_1(i_1,\ldots,i_{r-1}),\ldots,\hat{\phi}_s(i_1,\ldots,\phi_{r-1}),\ldots, \phi_{r-1}(i_1,\ldots ,\ldots,i_{r-1})\}}$, where $\mathbbm1_{\{a_1,\ldots,a_r\}}$ represents the indicator function of the set $$\{(a_{\sigma(1)},\ldots,a_{\sigma(r)})\in [n]^{r-1}: \sigma \text{ is a permutation of } [r-1]\},$$ the conditions on the $\phi_i$ imposed by the fact that $\phi$ is measure-preserving and bijective we obtain that  for a permutation $\sigma$ of $[r-1]$ we have $$\phi(i_1,\ldots,i_{r-1})=%(\psi\otimes \ldots \otimes \psi )(i_{\sigma(1)},\ldots,i_{\sigma(r-1)})=
(\psi(i_{\sigma(1)}),\ldots,\psi(i_{\sigma(r-1)}))$$
where $\psi$ is a permutation of $[n].$
\end{remark}

The previous theorem has the following direct important corollary:

\begin{corollary}
For two hypergraphs $H_1$ and $H_2$ with maximal edge cardinality $r$ the $(r-1)-$actions of their adjacency tensors $A(H_1)$ and $A(H_2)$ (that are $r-$th order tensors) are identified by the action convergence metric $d_M$ if and only if the hypergraphs $H_1$ and $H_2$ are isomorphic.
\end{corollary}

We expect that the previous theorem and remark can be generalized to any $s-$action ($s\in [r-1]$) of an $r-$th order tensor. The $1-$action case is trivial and we showed in the previous theorem and remark the $(r-1)-$action case.

\section{Sparse and non-uniform hypergraphs and different tensors}\label{NonUnifHypSect}
In this section, we study how one can use action convergence for sparse hypergraph sequences and for hypergraphs with different edge cardinalities (non-uniform hypergraphs), without losing information about edges with non-maximal cardinality. 

First of all, we discuss here how the sparseness of the hypergraphs interacts with our notions of action convergence.
We underline that the $2-$action for $3-$uniform hypergraphs might not be the best choice for sparser hypergraphs and the $1-$action might be sometimes more appropriate as the following example shows.

\begin{example}
Consider the $3-$uniform hypergraph $T(n,s_n)$ given by the triangles of the sparse Erdős–Rényi random graph $G(n,s_n)$ where $s_n \rightarrow 0 $ and $s_n n\rightarrow \infty$. For every $n$ we consider a realization $H_n$ of $T(n,s_n)$ and the related graph $G_n$ on the same vertex set with the hyperedges of $H_n$ as triangles. Let's denote with {\color{black}$E_n\subset [n]\times [n]$ the (symmetric) set} of edges of $G_n$ and recall that we denote with $A(H_n)$ the adjacency tensor of $H_n$. {\color{black} In this case, for every $f_n,g_n$ (sequences of) symmetric matrices, $\mathcal{L}\left(\left({A(H_n)}/{s_n}\right)_2[f_n,g_n]\right)$ weakly converges to $\delta_0,$ the Dirac function centered in $0.$} In fact, if we consider the sequence of multi-$P-$operators 
$$\left(\frac{A(H_n)}{s_n}\right)_2:(L^{\infty}([n]\times [n],Sym,\mathbb{P}_n))^2\longrightarrow L^{1}([n]\times [n],Sym,\mathbb{P}_n)$$
$$\left(\frac{A(H_n)}{s_n}\right)_2[f,g]_{i,k}=\frac{1}{2}(\sum^n_{j=1}\frac{A_{j,i,k}}{s_n}f_{j,i},g_{j,k} +\sum^n_{j=1}\frac{A_{j,i,k}}{s_n}f_{j,k},g_{j,i})
$$
where $\mathbb{P}_n$ is the uniform measure on $[n]\times [n]$, $\left(A(H_n)/{s_n}\right)_2[f,g]_{i,k}\neq 0$ if and only if $(i,k)\in E_n$. But as $n\rightarrow \infty$, $\mathbb{P}_n(E_n)\rightarrow 0$. For this reason, it might be appropriate to consider the $1-$action or change the probability measures $\mathbb{P}_n$ in such a way that $\mathbb{P}_n$ converges to some positive constant (for example choose $\mathbb{P}_n$ as the uniform probability measures on $E_n$).
\end{example}

Now, we present some possible choices to adapt action convergence to non-uniform hypergraphs.

In fact, considering the $2-$action (Definition \ref{DefAction}) associated with a (normalized) adjacency tensor of a hypergraph $H$ 

$$\widetilde{A}:(L^{\infty}([n]\times [n],Sym,\mathbb{P}_n))^2\longrightarrow L^{1}([n]\times [n],Sym,\mathbb{P}_n)$$
$$\widetilde{A}[f,g]_{i,k}=\frac{1}{2}(\sum^n_{j=1}\frac{A_{j,i,k}}{n}f_{j,i},g_{j,k} +\sum^n_{j=1}\frac{A_{j,i,k}}{n}f_{j,k},g_{j,i})
$$

we notice that considering the probability space $[n]\times [n]$ with uniform probability $\mathbb{P}_n$ (and the symmetric $\sigma-$algebra) the diagonal, i.e. the set 
$$
D_n=\{(i,i): i\in [n]\}\subset [n]\times [n]
$$
has probability $\mathbb{P}_n(D_n)=\frac{n}{n^2}=\frac{1}{n}$. Therefore, in the limit $n\rightarrow \infty$ we have that the edges of cardinality $2$ do not play any role in the profile measures of the multi-linear operator. However, we can choose other probability measures $\mathbb{P}_n$ different from the uniform distribution so that the information from the edges with lower cardinality is not lost. A natural choice for $\mathbb{P}_n$ is the discrete measure defined by $\mathbb{P}_n(\{(i,i)\})=\frac{1}{2n}$ and $\mathbb{P}_n(\{(i,j), (j,i)\}) =\frac{1}{2n(n-1)}$. This obviously characterizes uniquely the discrete probability measure $\mathbb{P}_n$. In this case, $\mathbb{P}_n(D_n)=\frac{1}{2}$ and, therefore, the lower cardinality edges play a role in the construction of the profiles and therefore of the limit object.

\begin{remark}
This construction of this  probability measure can be naturally generalized for the case $k>3$ where $\Omega=[n]^k$ with the symmetric $\sigma$-algebra.
\end{remark}

{\color{black} As simplicial complexes are a special case of general hypergraphs we obtain in such a way a notion of convergence for dense simplicial complexes.  Interest in a notion of convergence for dense simplicial complexes, similar to the one for dense graphs (graphons), has been expressed in  \cite{Bobrowski2022} describing it as a “potentially very interesting direction of future research in mathematics of random complexes”. Therefore, the study of this convergence and the relative limit objects might be of special interest. In \cite{RoddenCompl} the authors proposed a notion of limit for dense simplicial complexes, however, we remark that the counting lemma (Lemma 6) in \cite{RoddenCompl} cannot hold as stated (the proof is incorrect and a minor adaptation of the counterexamples for uniform hypergraphs, see \cite{HypergraphonsZhao}, gives a counterexample to the lemma).}

We have seen that we have different possible choices for the probability measures $\mathbb{P}_n$. We obviously have also many possible options for choosing different tensors and different normalizations of these tensors.

In fact, the (normalized) adjacency tensor is not the only tensor we can associate with a hypergraph. One possibility is to normalize dividing  every entry of the adjacency tensor by the quantity 

\begin{equation}\label{DegHyp}
\begin{aligned}
deg(i_1,\ldots, i_{k-1})=|\{e\in E \ \text{ s.t. }\ i_1,\ldots, i_{k-1}\in e  \\
\text{ and } |e|=|\{i_1,\ldots, i_{k-1}\}|+1
\end{aligned}
\end{equation} in the following way

$$
\widetilde{A}_{i_1,\ldots, i_{k}}=\frac{A_{i_1,\ldots, i_{k}}}{deg(i_1,\ldots, i_{k-1})}.
$$
It is easy to notice that $$deg(i_1,\ldots, i_{k-1})\leq |V|-k+1\leq |V|$$

In the particular case $k=3$ we have

$$
\widetilde{A}_{i,j,k}=\frac{A_{i,j,k}}{deg(i,k)}
$$

This is interesting for inhomogeneous hypergraphs and for hypergraphs with different edge cardinality.
In fact, we can define on $\Omega=[n]\times [n]$ the probability measure $$\mathbb{P}_n(\{(i,j)\})=\frac{deg(i,j)}{2\sum^n_{i,j=1,\ i\neq j}deg(i,j)}$$ if $i\neq j$ and $$\mathbb{P}_n(\{(i,i)\})=\frac{deg(i,i)}{2\sum^n_{i=1}deg(i,i)}.$$

These operators are also symmetric with respect to the right probability measure.

\begin{lemma}
    The operator $(\widetilde{A})_2$  is symmetric with respect to the probability measure $\mathbb{P}_n$. 
\end{lemma}
\proof
The lemma follows from the following equality
$$\begin{aligned}
&\mathbb{E}[(\widetilde{A})_2[f,g]h]=\\
&\frac{1}{2}(\sum^n_{i,k=1, \ i\neq k}(\sum^n_{j=1}\frac{A_{i,j,k}}{deg(i,k)}f_{i,j}g_{j,k}+\sum^n_{j=1}\frac{A_{i,j,k}}{deg(i,k)}g_{i,j}f_{j,k}))h_{i,k}\frac{deg(i,k)}{2\sum^n_{i,k=1,\ i\neq k}deg(i,k)}+\\
& \frac{1}{2}(\sum^n_{i}(\sum^n_{j=1}\frac{A_{i,j,i}}{deg(i,i)}f_{i,j}g_{j,i}+\sum^n_{j=1}\frac{A_{i,j,i}}{deg(i,i)}g_{i,j}f_{j,i}))h_{i,i}\frac{deg(i,i)}{2\sum^n_{i=1}deg(i,i)})=\\&
\frac{1}{2}(\sum^n_{i,k=1, \ i\neq k}(\sum^n_{j=1}A_{i,j,k}f_{i,j}g_{j,k}+\sum^n_{j=1}A_{i,j,k}g_{i,j}f_{j,k}))h_{i,k}\frac{1}{2\sum^n_{i,k=1,\ i\neq k}deg(i,k)}+\\
& \frac{1}{2}(\sum^n_{i}(\sum^n_{j=1}A_{i,j,i}f_{i,j}g_{j,i}+\sum^n_{j=1}A_{i,j,i}g_{i,j}f_{j,i}))h_{i,i}\frac{1}{2\sum^n_{i=1}deg(i,i)})=\\&
\mathbb{E}[(\widetilde{A})_2[f,h]g]=\mathbb{E}[(\widetilde{A})_2[h,g]f].
\end{aligned}
$$\endproof
Therefore, the limit of a sequence of such operators will be also symmetric and positivity-preserving by Lemma \ref{LemmaClosedSymmetric} and Lemma \ref{LemmaClosedPosPres}.

\begin{remark}
    The previous lemma can be easily generalized for the case $k>3.$
\end{remark}

\section{Multi-action convergence, hypergraphons and P-variables}\label{SectHypergraphons}

From Theorem 8.2 and Lemma 8.3 in \cite{backhausz2018action} we have that dense simple graph sequences convergence (convergence in real-valued cut distance $\delta_{\square,\R}$) is equivalent to the action convergence of the sequence of the normalized adjacency matrices 

$$
\frac{A(G_n)}{|V(G_n)|}.
$$
and to the action convergence of real-valued graphons.

In this section, we present some ideas on the connection of multi-action convergence and other hypergraph limits for dense hypergraph sequences.

The theory of dense $r-$uniform hypergraph limits (hypergraphons) has been developed in \cite{hypergrELEK20121731} using techniques from model theory (ultralimits, ultraproducts) and successively translated in a more standard graph limit  language in \cite{HypergraphonsZhao}. A good presentation of the model-theoretic approach is given in \cite{RandomneLimitTow}. We briefly present here the theory of dense hypergraph limits, highlighting the similarities with action convergence, following the analytic presentation in \cite{HypergraphonsZhao}. 

We start with some notation. For any subset $A\subset [n]$, define $r(A)$ to be the collection of all nonempty subsets of $A$, and $r_{<}(A)$ to be the collection of all nonempty proper subsets of $A$. More generally, let $r(A, m)$ denote the collection of all nonempty subsets of $A$ of size at most $m$. So for instance, $r_{<}([k])=r([k], k-1)=r([k]\setminus \{k\})$. We will also use the shorthand $r[k]$ and $r_{<}[k]$ to mean $r([k])$ and $r_{<}([k])$ respectively.

Any permutation $\sigma$ of a set $A$ induces a permutation on $r(A, m)$. For a set $A=\{v_1,\ldots v_t\}\subset [k]$ of cardinality $t$ where $v_1<\ldots <v_t$, we indicate with $\mathrm{x}_A=(x_{v_1},\ldots,x_{v_t},x_{v_1v_2}\ldots,x_{v_1\ldots v_t} )$.   

The limit object of a sequence of $r-$uniform hypergraphs, i.e.\ an {\color{black} $r-$h}ypergraphon, is a symmetric measurable function
{\color{black}
$$
W:[0,1]^{2^r-2}\longrightarrow [0,1].
$$
$$
W(\mathrm{x}_{r[r]})=W(x_1,\ldots, x_r,x_{12},\ldots,x_{(r-1)r},\ldots x_{12\ldots r-1},\ldots,x_{2\ldots r})
$$
where symmetric means that 
\begin{equation*}
\begin{aligned}
   & W(x_1,\ldots, x_r,x_{12},\ldots,x_{(r-1)r},\ldots x_{12\ldots (r-1)},\ldots,x_{2\ldots r})=\\
  &  W(x_{\sigma(1)},\ldots, x_{\sigma{(r)}},x_{\sigma{(1)}\sigma{(2)}},\ldots,x_{\sigma(r-1)\sigma(r)},\ldots x_{\sigma(1)\sigma(2)\ldots \sigma(r-1)},\ldots,x_{\sigma(2)\ldots \sigma(r)})
    \end{aligned}
\end{equation*}

for every permutation $\sigma$ of $[r]$.} This might be surprising because, differently from the case of graphs ($r=2$), for $r>2$ the dimensionality of the $r-$th order adjacency tensor associated to an $r-$uniform hypergraph, $r$,  does not coincide with the dimensionality of the {\color{black} $r-$}hypergraphon, $2^r-2$.

The need for the additional coordinates, representing all proper subsets of $[r]$, is related to the need for suitable regularity partitions for hypergraphs \cite{GowersHypRegularity,RodlHyperReg1, RodlHyperReg2}  and it is moreover related to the hierarchy of notions of quasi-randomness in the case of $r-$uniform hypergraphs for $r>2$  \cite{RandomneLimitTow}.

This is also intuitively related to the various multi-$P$-operators associated with a tensor through its $s-$actions. In fact for $r=3$ the additional coordinates are again needed, for example, to differentiate the limits of the sequence of the Erdős–Rényi $3-$uniform hypergraphs $G(n,\frac{1}{8},2)$ (Example \ref{ERRandomHypergraph}) and the sequence of the $3-$uniform hypergraphs $T(n,\frac{1}{2})$ given by the triangles of the Erdős–Rényi graph (Example \ref{TriangRandHyp}).

We notice that similarly to how we associated graphons to $P-$operators we can associate hypergraphons to multi-$P-$operators:
{\color{black}

$$
\widehat{W}:L^{\infty}([0,1]^{2^{r-1}-2},Sym)\times \ldots \times L^{\infty}([0,1]^{2^{r-1}-2},Sym)\longrightarrow L^1([0,1]^{2^{r-1}-2},Sym)
$$
\begin{equation}\label{EqOpHypergraphon}
\begin{aligned}
\widehat{W}[g_1,\ldots, g_{(r-1)}]&(\mathrm{x}_{r([r]\setminus \{r\})})\\&=\frac{1}{(r-1)!}\sum_{\sigma}\int_{[0,1]^{2^r-2^{r-1}+1}}W(\mathrm{x}_{r[r]})\prod^{r-1}_{i=1} g_{\sigma(i)}(\mathrm{x}_{r_{[r]\setminus\{i\}}})\mathrm{d}\mathrm{x}_{A(r)}
\end{aligned}
\end{equation}
where $\sigma$ here is a permutation of $[r-1],$ $A(r)$ is the set of all the proper subsets of $[r]$ containing $r,$ and $Sym$ is the symmetric $\sigma-$algebra (i.e.\ the $\sigma-$algebra generated by the subsets of  $[0,1]^{2^{r-1}-2}$ that are invariant under the action of all permutations of $[r-1]$).
}
{\color{black}In particular, for $r=3,$ we have \begin{equation*}
\begin{aligned}
&\widehat{W}[g^{(1)}, g^{(2)}](x_1,x_2,x_{12})\\&=\frac{1}{2}\int_{[0,1]^{4}}W(x_1,x_2,x_3,x_{12},x_{13},x_{23})g^{(1)}(x_1,x_3,x_{13})g^{(2)}(x_2,x_3,x_{23})\mathrm{d}x_3\mathrm{d}x_{13}\mathrm{d}x_{23}\\
&+\frac{1}{2}\int_{[0,1]^{4}}W(x_1,x_2,x_3,x_{12},x_{13},x_{23})g^{(2)}(x_1,x_3,x_{13})g^{(1)}(x_2,x_3,x_{23})\mathrm{d}x_3\mathrm{d}x_{13}\mathrm{d}x_{23}\\
\end{aligned}
\end{equation*}

We observe that there are promising similarities between the action convergence of hypergraphons and the action convergence of the $(r-1)-$action of the adjacency tensor.

Let's consider for example the hypergraphon,
$$
W(x_1,x_2,x_3,x_{12},x_{13},x_{23})=\begin{cases}
    1 \ \text{ if } 0\leq x_{12},x_{13},x_{23}\leq \frac{1}{2}\\
    0 \ \text{ else}
\end{cases}
$$

that is the limit of the sequence of hypergraphs $T(n,\frac{1}{2})$ given by the triangles of the Erdős–Rényi random graph  (see \cite{HypergraphonsZhao} for example) and the action convergence limit of the $2-$action $(B_n)_2$ of the sequence of tensors $(B_n)_n$ obtained normalizing the adjacency tensors of the same hypergraphs, i.e.\ $B_n=\frac{A(T(n,\frac{1}{2}))}{n}$ (recall Example \ref{ExamCompERHyp}),  we have, for example, that
$$ 
\mathcal{L}(\mathbbm{1}_{\Omega_n},\mathbbm{1}_{\Omega_n},(B_n)_2[\mathbbm{1}_{\Omega_n},\mathbbm{1}_{\Omega_n}])\rightarrow  \frac{1}{2}\delta_{(1,1,0)}+\frac{1}{2}\delta_{(1,1,\frac{1}{4})}=\mathcal{L}(\mathbbm{1}_{\Omega},\mathbbm{1}_{\Omega},\widehat{W}[\mathbbm{1}_{\Omega},\mathbbm{1}_{\Omega}])
$$ also if we take the set $S_n$ to be the {\color{black}(symmetric) set of}pairs that correspond to edges of $G(n,\frac{1}{2})$. Then also   
$$ \mathcal{L}(\mathbbm{1}_{S_n},\mathbbm{1}_{S_n},(B_n)_2[\mathbbm{1}_{S_n},\mathbbm{1}_{S_n}])\rightarrow  \mathcal{L}(\mathbbm{1}_{[0,1]^2\times [0,\frac{1}{2}]},\mathbbm{1}_{[0,1]^2\times [0,\frac{1}{2}]},\widehat{W}[\mathbbm{1}_{[0,1]^2\times [0,\frac{1}{2}]},\mathbbm{1}_{[0,1]^2\times [0,\frac{1}{2}]}])
$$
and, similarly,

$$\mathcal{L}(\mathbbm{1}_{S^c_n},\mathbbm{1}_{S^c_n},(B_n)_2[\mathbbm{1}_{S^c_n},\mathbbm{1}_{S^c_n}])\rightarrow \mathcal{L}(\mathbbm{1}_{[0,1]^2\times [\frac{1}{2},1]},\mathbbm{1}_{[0,1]^2\times [\frac{1}{2},1]},\widehat{W}[\mathbbm{1}_{[0,1]^2\times [\frac{1}{2},1]},\mathbbm{1}_{[0,1]^2\times [\frac{1}{2},1]}])$$
and 

$$\mathcal{L}(\mathbbm{1}_{S_n},\mathbbm{1}_{S^c_n},(B_n)_2[\mathbbm{1}_{S_n},\mathbbm{1}_{S^c_n}])\rightarrow \mathcal{L}(\mathbbm{1}_{[0,1]^2\times [0,\frac{1}{2}]},\mathbbm{1}_{[0,1]^2\times [\frac{1}{2},1]},\widehat{W}[\mathbbm{1}_{[0,1]^2\times [0,\frac{1}{2}]},\mathbbm{1}_{[0,1]^2\times [\frac{1}{2},1]}]).$$

{\color{black}
Moreover, for any two $3-$hypergraphons $W$ and $U$ we can consider the multi-action convergence metric $d_M$ between the associated multi-$P-$operators $\widehat{W}$ and $\widehat{U}$ defined in equation \eqref{EqOpHypergraphon}. In particular, in this case, for the multi-$P-$operators $\widehat{W}$, equation \eqref{EqConstructMultiAct} in the construction of the action convergence metric is
$$
\begin{aligned}
&\mathcal{L}(g^{(1)}_1,\ldots, g^{(1)}_k,g^{(2)}_1,\ldots, g^{(2)}_k,\widehat W[g_1^{(1)},g_1^{(2)}],\ldots,\widehat W[g_k^{(1)},g_k^{(2)}]).
\end{aligned}$$
where, for $j\in [k],$ we consider $g^{(1)}_j,g^{(2)}_j\in L_{[-1,1]}^{\infty}(\Omega_1\times \Omega_1 \times \Omega_2).$

From Lemma \ref{limspdm}  and Lemma \ref{LemmCutInf1} we also obtain the following estimate.

\begin{lemma}
For any two $3-$hypergraphons $W$ and $U$ and the associated multi-$P-$operators $\widehat{W}$ and $\widehat{U}$ defined in equation \eqref{EqOpHypergraphon} we have the following inequality $$\begin{aligned}
    d_M(\widehat{W},\widehat{U})&\leq {\color{black} 12}(\|W-U\|_{\square_{2}})^{1/2}
\end{aligned}$$
where  for a linear combination of $3-$hypergraphons $V$ 
$$\begin{aligned}\|V\|_{\square_{2}}=\sup_{g_1,g_2,g_3}&\left|\int_{[0,1]^6}V(x_1,x_2,x_3,x_{12},x_{13},x_{23})g_1(x_1,x_2,x_{12})g_2(x_2,x_3,x_{23})g_3(x_1,x_3,x_{13})\mathrm{d}x_1\right.\\
    &\left.\mathrm{d}x_2\mathrm{d}x_3\mathrm{d}x_{12}\mathrm{d}x_{13}\mathrm{d}x_{23}\right|.\end{aligned}$$ where the supremum is taken over measurable $g_i:[0,1]^3\rightarrow [0,1]$ for every $i\in [3]$ such that $g_i(x_1,x_2,x_{12})=g_i(x_2,x_1,x_{12}).$
\end{lemma}

\begin{remark}
    More generally for two $r-$hypergraphons $W$ and $U$ we have the following bound for the multi-action convergence distance $d_M$ between the multi-$P-$operators $\widehat{W}$ and $\widehat{U}$ defined in equation \eqref{EqOpHypergraphon}:
    \begin{equation*}
    d_M(\widehat{W},\widehat{U})\leq {\color{black} 3 \cdot 2^{r-1}}(\|W-U\|_{\square_{r-1}})^{1/2}
\end{equation*}

where $\|\cdot\|_{\square_{r-1}}$ is the $(r-1)-$cut norm from Definition 4.3 in \cite{HypergraphonsZhao}.
\end{remark}

In particular, we obtain the following corollary from the previous lemma and remark.

\begin{corollary}
Hypergraphon convergence in the sense of Definition 6.6 of \cite{HypergraphonsZhao} (Partitionable convergence) implies action convergence of hypergraphons (interpreted as multi-$P-$operators as in \eqref{EqOpHypergraphon}). Moreover, the limits have to be compatible.     
\end{corollary}

{\color{black}

We anticipate a deeper connection between multi-action convergence, $P-$variables convergence (see Section 9.4 in \cite{zucal2024probabilitygraphonspvariablesequivalent}) and convergence of hypergraphons (Definition 6.6 in \cite{HypergraphonsZhao}) that we will explore in future work. 

{\color{black} We briefly sketch some motivating ideas here.

Let's denote $\Omega_1=\Omega_2=[0,1]$ for every $i\in [6].$ Let $W$ be a hypergraphon and $\widehat W$ its multi-$P-$operator representation. Observe, in particular, that we can construct also sets of measures, similarly to as done in Section \ref{SecMultiActConv} (see equation \eqref{EqConstructMultiAct}), constructing this time probability measures out of the random vectors $Y$ from $[0,1]^3=\Omega_1\times \Omega_1 \times \Omega_2$ to $\R^{7k}$

$$
\begin{aligned}
&Y(x_1,x_2,x_{12})\\
&=(f^{(1)}_1(x_1),f^{(1)}_1(x_2),\ldots, f^{(1)}_k(x_1),f^{(1)}_k(x_2),
g^{(1)}_1(x_1,x_2,x_{12}),\ldots, g^{(1)}_k(x_1,x_2,x_{12})\\
&\hspace{0.4 cm}f^{(2)}_1(x_1),f^{(2)}_1(x_2),\ldots,f^{(2)}_k(x_1),f^{(2)}_k(x_2) ,g^{(2)}_1(x_1,x_2,x_{12}),\ldots, g^{(2)}_k(x_1,x_2,x_{12})\\
&\hspace{0.4 cm}\widehat W[g_1^{(1)},g_1^{(2)}](x_1,x_2,x_{12}),\ldots,\widehat W[g_k^{(1)},g_k^{(2)}](x_1,x_2,x_{12})).
\end{aligned}$$
where, for $j\in [k],$ we consider $g^{(1)}_j,g^{(2)}_j\in L_{[-1,1]}^{\infty}(\Omega_1\times \Omega_1 \times \Omega_2)$ as before and we additionally consider $f^{(1)}_j,f^{(2)}_j\in L^{\infty}_{[-1,1]}(\Omega_1).$ 

Therefore, one can also define a metric for hypergraphons considering the Hausdorff metric on the space of measures as in Section \ref{SecMultiActConv}, recall Definition \ref{DefActMetric}. We observe that this metric works well only for dense hypergraph sequences. We expect this convergence to be equivalent to hypergraphon convergence (as defined in Definition 6.6 in \cite{HypergraphonsZhao}). Notably, this sketched convergence trivially implies multi-action convergence for hypergraphons. Specifically, if the action convergence limits of two sequences of hypergraphons differ, the limits under this modified convergence will also differ. Therefore, we expect action convergence to serve as a useful benchmark for understanding hypergraphon convergence. We have demonstrated many desirable properties for action convergence, which suggests (in some cases directly implies) that these properties also apply to the alternative convergence described above. }

Moreover, the convergence just outlined can be viewed as a contraction of the extension of $P$-variables to hypergraphs (as discussed in Section 9.4 in \cite{zucal2024probabilitygraphonspvariablesequivalent}). Recall that in the case of real-valued graphons, action convergence is equivalent to convergence in the real-valued cut distance, which can be considered a contraction of the $P$-variable metric, {\color{black} see Definition 4.19, Corollary 7.9.1 and Lemma 7.9 in \cite{zucal2024probabilitygraphonspvariablesequivalent}} (or equivalently, the unlabelled cut distance for probability graphons, {\color{black} see also \cite{abraham2023probabilitygraphons} and \cite{zucal2024probabilitygraphonsrightconvergence})}. 

As already said, we will compare these convergence notions in detail in future work. {\color{black} We expect/conjecture the equivalence of the convergence formulated by Yufei Zaho (Definition 6.6 in \cite{HypergraphonsZhao}) and the modified version of action convergence sketched above for hypergraphons.}}

}}

\section*{Appendix (technical lemmas)}\label{SecAppendixTech}

For completeness, we collect here a series of lemmas proven in  \cite{backhausz2018action} that we used extensively throughout our work.

We start with an upper-bound on the Lévy–Prokhorov distance of the distribution of two random variables

\begin{lemma}[Lemma 13.1 in \cite{backhausz2018action}]\label{coupdist} Let $X,Y$ be two jointly distributed $\mathbb{R}^k$-valued random variables. Then $$d_{\mathcal{LP}}(\mathcal{L}(X),\mathcal{L}(Y))\leq \tau(X-Y)^{1/2}k^{3/4},$$

where $\tau$ is defined as in \eqref{eqn:tau}.
\end{lemma}

A direct consequence of the previous statement is the following Lemma.

\begin{lemma}[Lemma 13.2 in \cite{backhausz2018action}\label{coupdist2}] Let $v_1,v_2,\dots,v_k$ and $w_1,w_2,\dots,w_k$ be in $L^1(\Omega)$ for some probability space $\Omega$. Let $m:=\max_{i\in [k]} \|v_i-w_i\|_1$. Then
$$d_{\mathcal{LP}}(\mathcal{L}(v_1,v_2,\dots,v_k),\mathcal{L}(w_1,w_2,\dots,w_k))\leq m^{1/2}k^{3/4}.$$
\end{lemma}

The next lemma is a general probabilistic result about limits of random variables, products and expectations.

\begin{lemma}[Lemma 13.4 in \cite{backhausz2018action}]\label{closedlem2} Let $q\in (1,\infty)$.  Let $\{(X_i,Y_i)\}_{i=1}^\infty$ be a sequence of pairs of jointly distributed real-valued random variables such that $X_i\in [-1,1]$ and $\mathbb{E}[|Y_i|^q]\leq c<\infty$ for some $c\in\mathbb{R}^+$. Assume that the distributions of $(X_i,Y_i)$ weakly converge to some probability distribution $(X,Y)$ as $i$ goes to infinity. Then $\mathbb{E}[|Y|^q]\leq c$ and $$\lim_{i\to\infty} \mathbb{E}[X_iY_i]=\mathbb{E}[XY].$$ 
\end{lemma}

We give a last technical upper bound for the Lévy–Prokhorov distance of measures generated by a $P-$operator through specific random variables. This is a minor modification of Lemma 13.6 in \cite{backhausz2018action}.

\begin{lemma}\label{applem2} Let $p\in [1,\infty)$ and let $A\in\mathcal{B}_r(\Omega)$ be a multi-$P$-operator. Let  $v_i$ and $w_i$ be functions in $L_{[-1,1]}^\infty(\Omega)$ for every $i\in [k]$. Then we have
$$d_{\mathcal{LP}}(\mathcal{D}_A(\{v_i\}_{i=1}^k),\mathcal{D}_A(\{w_i\}_{i=1}^k))\leq m^{1/2}((2d)^p+2^{p+1}d)^{1/(2p)}(2k)^{3/4},$$ where $m=\max\{1,(r-1)\|A\|_{p\ldots,p\to 1}\}$ and $d=\max_{i\in [k]}\{d_{\mathcal{LP}}(\mathcal{D}(v_i-w_i),\delta_0)\}$.
\end{lemma}
\proof
The proof is identical to the proof of Lemma 13.6 in \cite{backhausz2018action}, except that we use the properties of the multi-linear norm here.
\endproof

\textbf{Acknowledgements:} The author thanks Ágnes Backhausz, Tobias B\"ohle, Christian K\"uhn, Raffaella Mulas, Florentin M\"unch, Balázs Szegedy, {\color{black} Sjoerd van der Niet} and Chuang Xu for useful discussions. {\color{black} This work is part (in a slightly different form) of the author's PhD thesis.}
\medskip

\section*{References}

\bibliographystyle{plain}
\bibliography{biblio1}

\end{document}